\documentclass[a4paper,times,doublespace,10pt]{article}
\usepackage{epsfig}
\usepackage{graphicx}
\usepackage{amsmath}
\usepackage{amssymb}
\usepackage{caption}
\usepackage{enumitem}

\newcommand\norm[1]{\left\lVert#1\right\rVert}
\newcommand\abs[1]{\left\lvert#1\right\rvert}
\DeclareMathOperator\supp{supp}

\DeclareMathOperator{\diam}{diam}

\def\bbR{\mathbb{R}}

\def\bbNz{\mathbb{N}_0}
\def\bbN{\mathbb{N}}

\def\bbP{\mathbb{P}}
\def\bbE{\mathbb{E}}
\def\bbQ{\mathbb{Q}}
\def\cC{\mathcal{C}}
\def\cB{\mathcal{B}}
\def\cF{\mathcal{F}}
\def\cP{\mathcal{P}}
\def\cQ{\mathcal{Q}}
\def\cM{\mathcal{M}}

\def\cL{\mathcal{L}}

\def\tn{\tilde{n}}

\def\bS{{\bm S}}

\def\alpt{\Tilde{\alpha}}

\def\NFKX_2{NF-$\kappa$X_2 }
\def\IKX_2A{I$\kappa$X_2$\alpha$ }

\def\k0{k_*}

\def\vec1{\text{vec}}

\def\NFKX_2{NF-$\kappa$X_2 }
\def\S0{\bS_0}

\def\Poi{\hbox{Poi}}
\def\CLE{\text{\tiny CLE}}
\def\LNA{\text{\tiny LNA}}

\usepackage{amsthm}

\newtheorem{theorem}{Theorem}[section]
\newtheorem{theoremx}{Theorem}

\newtheorem{corollary}[theorem]{Corollary}
\newtheorem{lemma}[theorem]{Lemma}
\newtheorem{proposition}[theorem]{Proposition}
\newtheorem{definition}[theorem]{Definition}

\newtheorem{notation}[theorem]{Notation}

\newtheorem*{example}{Example}

\usepackage{authblk}
\title{Stability and synchronisation in modelling an oscillatory stochastic reaction network
}
\author{Frederick Truman-Williams}
\affil{Mathematical Institute, University of Oxford, Oxford, UK}
\date{}                  
\begin{document}
\maketitle

\begin{abstract}
    In many applied settings, the chemical Langevin equation and linear noise approximation are used in the simulation and data analysis of stochastic reaction networks. With the goal of exploring the sensitivities of reaction network paths to their initial conditions, we subject these modelling techniques to the analysis of random dynamical systems and stochastic flows of diffeomorphisms respectively. After introducing this perspective to stochastic reaction networks in general, we turn our attention to the Brusselator: a two dimensional stochastic reaction network whose paths, when noise is neglected, exhibits a Hopf bifurcation. Studying both Lyapunov exponents, as well as their finite time counterparts, provides two new insights. Firstly, the Brusselator, when modelled by the chemical Langevin equation, exhibits a global synchronisation property of paths of similar noise realisations; secondly, contrary to statistical accuracy in the distributions of concentrations of reactants, the linear noise approximation can fail to capture the finite time dynamical properties of paths of the chemical Langevin equation. In doing so, we explore the notions of dynamical bifurcation and quasi-ergodicity.
\end{abstract}

\section{Introduction}\label{introdction}

Consider a well-stirred system of $d \geqslant 1$ molecular species $\{S_1,...,S_d\}$ that chemically interact within some fixed volume of constant temperature through $M \geqslant 1$ reaction channels $\{ R_1,\dots,R_r \}$. We specify the dynamical state of this system by its state vector:
\begin{definition}[State vector]
For each molecular species ${S_1,...S_d}$ let $Y_i:\bbR_{\geq0} \to \bbNz, \ t\mapsto Y_i(t)$ denote the number of $S_i$ molecules at time $t$, $i\in\{1,...,d\}$. We refer to the \textit{system state} or \textit{state vector} as $Y:[0,\infty) \to \bbNz^d, \ t\mapsto Y(t) \equiv (Y_1(t),...,Y_d(t))^\intercal .$ 
\end{definition}
Here, the molecular population $Y_i(t)$ at time $t\geq 0$ will be a random variable on some probability space with probabilty measure $\bbP$. We seek to describe the evolution of $Y(t)$ from some given initial state $Y(t_0) = y_0$ subject to reactions $R_1,...,R_r$ occurring over and over at random times. We refer to reactions as instantaneous events which change the number of molecules of each specie. We also assume that two reactions cannot occur at exactly the same time. In order to specify a reaction occurring at time $t$ , say $R_j$ for some $j\in\{1,..,r\}$, we make the following definition.

\begin{definition}[State-change vector]
Consider a well-stirred fixed temperature system of $\{S_1,...,S_d\}$ molecular species and $\{ R_1,\dots,R_r \}$ reaction channels with state vector $Y$. The \textit{stoichiometry matrix} is the matrix $\nu \in \bbNz^{d\times r}$ where the $(i,j)^\textsuperscript{th}$ entry corresponds to the change in the number of $S_i$ molecules produced by one $R_j$ reaction. We refer to each column of the stoichiometry matrix $\nu_j$ as the $j^\textsuperscript{th}$ reaction \textit{state-change vector}.
\end{definition}
Now the updated the system state at time $t\geq 0$ after the $R_j$ reaction channel fires is $Y(t)+\nu_j$. This evolution law implies $\{Y(t) \ | \ t\geq 0\}$ is a jump-type process on the non-negative $d$-dimensional integer lattice.

Under well-stirred conditions, it can be shown that for each reaction channel $R_j$ generates a well defined reaction rate function, independent of the molecules positions or velocities.

\begin{definition}[Propensity function]\label{def:prop}
Consider a well-stirred fixed temperature system of $\{S_1,...,S_d\}$ molecular species and $\{ R_1,\dots,R_r \}$ reaction channels with state vector $Y$. Under these conditions, there exists functions $\alpha_j:\bbNz^d \to [0,\infty)$ for each $j\in\{1,..,r\}$ with property
\[ \bbP(\{Y(t+s)-Y(t) = \nu_j\}) = \alpha_j(Y(t))s + o(s), \]
for all $s,t\geq 0$. The function $\alpha_j$ is known as the propensity function for reaction $R_j$.
\end{definition}

Intuitively, one can think of the propensity function $\alpha_j$ to be the function where the probability, given $Y(t)=y$, that one $R_j$ reaction will occur in the next infinitesimal time interval $[t,t+dt)$ is given by $\alpha_j(y)dt$. The propensity function $\alpha_j$ along with the state-change vector $\nu_j$ completely specifies the reaction channel $R_j$.

Under these conditions, the time-evolution of the stochastic process $\{Y(t) \ | \ t\geq 0\}$ can be described using the random time change representation (RTC) \cite{Anderson2011}
\begin{equation}\label{eq:rtc}
	Y(t) = Y(0) + \sum_{j=1}^r \Poi_j \left( \int_0^t \alpha_j(Y(s)) \
ds  \right)  \nu_j,
\end{equation} 
where, for $j=\{1,2,\dots,r\}$, the $\Poi_j$ are independent, Poisson
processes of rate $\int_0^t \alpha_j(Y(s)) \
ds$ corresponding to reaction $R_j$. Engel in \cite{Mendez2024} analysed some dynamical properties of such paths for the special setting of birth-death reaction networks.

Exact realisations of $\{Y(t) \ | \ t\geq 0\}$ can be obtained using Gillespie's stochastic simulation algorithm (SSA) \cite{Gill77}. However in practical applications such as biochemical processes \cite{ciocchetta2010,Minas2017}, the computational cost of using the SSA to simulate models or provide inference of data is far too large to justify its use. To combat this, a vast amount of literature has been produced on approximate methods to simulate paths of $\{Y(t) \ | \ t\geq 0\}$ and their validity. 

To this end, it is common to introduce a system size parameter $\Omega$. Through setting $X(t)\equiv Y(t)/\Omega$, one can study the dependence of stochastic fluctuations upon system size. It is sufficient to assume that the rates $\alpha_j(Y(t))$ depend upon $\Omega$ as $\alpha_j(Y(t)) \equiv \Omega \alpt_j(X(t))$ where $\alpt_j$ are called the macroscopic rates of the system. This is a fairly-relaxed condition, that can be further relaxed to more general relations \cite{Kampen1992,kurtz1981approximation}. Upon substitution into \eqref{eq:rtc}, $\{X(t) \ | \ t\geq 0\}$ evolves according to,
\begin{equation}\label{eq:rtc2}
	X(t) = X(0) + \frac{1}{\Omega}\sum_{j=1}^r \Poi_j \left( \Omega \int_0^t \alpt_j(X(s)) \
ds  \right) \nu_j.
\end{equation} 
For the remainder of the paper we will investigate approximate models of $\{X(t) \ | \ t\geq 0\}$ and so when we refer to the notion of accuracy in such models we mean in relation to the distribution of $\{X(t) \ | \ t\geq 0\}$ as it evolves under \eqref{eq:rtc2}. In the literature, a standard first approximation of $\{X(t) \ | \ t\geq 0\}$ for sufficiently large $\Omega$ is the process $\{X^\CLE(t) \ | \ t\geq 0\}$ which evolves according to the chemical Langevin equation (CLE) \cite{Anderson2011}:
\begin{equation}\label{eq:CLE} 
    X^\CLE(t) = X^\CLE(0) + \sum_{j=1}^r \nu_j \int_0^t\alpt_j(X^\CLE(s)) \ ds + \frac{1}{\sqrt{\Omega}} \sum_{j=1}^r \nu_j \int_0^t \sqrt{\alpt_j(X^\CLE(s))} \ dW^j_s ,
\end{equation}
where, for $j\in\{1,2,\dots,r\}$, the $W^j_t$ are independent Wiener processes. In It\^{o} SDE form, the CLE \eqref{eq:CLE} is open to more accessible analytical techniques and much faster simulation. Moreover, in \cite{Gill2000,Anderson2011} it was shown that for reaction networks where molecular populations usually don't drop to near zero, distributions of $\{X^\CLE(t) \ | \ t\geq 0\}$ are nearly indistinguishable from that of $\{X(t) \ | \ t\geq 0\}$ for not very large $\Omega$. This situatation is typical in many biochemical oscillatory processes \cite{LEFEVER1971267,ciocchetta2010} one of which will the case we study in this paper.

In pursuit of further analytical tractability and even faster simulation, van Kampen \cite{Kampen1992, kurtz1981approximation} forced a solution \eqref{eq:CLE} of the form
\begin{equation}\label{LNAansatz}
    X^\LNA(t) = z(t) + \Omega^{-1/2}\xi(t),
\end{equation}

where $z(t)$ is the limit in probability of $X(t)$, i.e.\ $ X(t)\stackrel{P}{\to} z(t) $, as $\Omega \to \infty$ and $\xi(t)$ are (random) pertubations from $z(t)$. In his method, called the linear noise approximation (LNA) the Kolmogorov forward equation governing the distribution of $\{X(t) \ | \ t\geq 0\}$ is approximated by a Fokker-Planck equation with linear coefficients resulting in an explicit distribution for $\{\xi(t) \ | \ t\geq0\}$ (and hence $\{X^\LNA(t) \ | \ t\geq 0\}$). That is $\{\xi(t) \ | \ t\geq0\}$ solves the linear stochastic equation,
\begin{equation}\label{eq:xiSDE}
    \xi(t) =  \xi(0) + \sum_{j=1}^r \nu_j \int_0^t\mathrm{D}_z \alpt_j(z(s))\xi(s) \ ds + \sum_{j=1}^r \nu_j \int_0^t\sqrt{\alpt_j(z(s))} \ dW^j_s,
\end{equation}
where $\mathrm{D}_z$ denotes the derivative operator with respect to $z$, for $j\in\{1,2,\dots,r\}$, the $W^j_t$ are independent Wiener processes and $\{ z(t) \ | \ t\geq0\}$ solves the reaction rate ODE (RRE) for some initial value $z(0) = z_0\in\bbR^d$,
\begin{equation}\label{eq:RRE}
    \frac{dz}{dt} = \sum_{j=1}^r \alpt_j(z)\nu_j, \quad z(0) = z_0.
\end{equation}
By solving \eqref{eq:xiSDE}, we obtain,
\begin{equation}\label{eqn:xistate}
    \xi(t) \sim \text{MVN}( C(0,t) \xi(0), V(0,t)) \quad \text{for each } t\geq 0,
\end{equation}
where MVN$( C(0,t) \xi(0), V(0,t))$ denotes the multivariate normal distribution with mean $C(0,t) \xi(0)$ and covariance matrix $V(0,t)$. Here $C(0,t)\in\bbR^{d\times d}$ is the fundamental matrix of \eqref{eq:RRE}, i.e.\
the solution of the initial value problem
\begin{equation}\label{eq:Code}
   \frac{dC}{dt} = \sum_{j=1}^r \nu_j\mathrm{D}_z\alpt_j(z) C, \quad C(0,0) = I_d.
\end{equation}
and $V(0,t)$ the symmetric, positive-definite matrix which is a solution of the initial value problem
\begin{equation}\label{eq:Vode}
    \frac{dV}{dt} = \sum_{j=1}^r \nu_j \mathrm{D}_z\alpt_j(z) V + \sum_{j=1}^r V \left(\mathrm{D}_z\alpt_j(z)\right)^\intercal \nu_j^\intercal + \sum_{j=1}^r \nu_{j}\nu_j^\intercal\alpt_j(z), \quad V(0,0) = 0_d.
\end{equation} 

The LNA is far superior in speed of simulation and allows rapid parameter estimation and statistical inference; something the CLE cannot provide to any degree. Though the LNA has shown long-time accuracy for networks involving propensity functions that are up to first order polynomials of the reactants concentrations \cite{Thomas2013}, the nonlinear dynamics of $\{X(t) \ | \ t\geq 0\}$ are lost \cite{Wallace2012}, proving detrimental to approximate networks exhibiting oscillatory dynamics unless modelled over short finite times. In order to reap the benefits of the LNA for networks with sustained oscillations, Minas and Rand \cite{Minas2017,Minas2024} developed a way to use the LNA many times during the simulation of a single path, ensuring the LNA was only making predictions over short finite times and hence maintaining accuracy for long times. 

Soon, we will explore some of the dynamical properties of the paths of approximations $\{X^\CLE(t) \ | \ t\geq 0\}$ and $\{X^\LNA(t) \ | \ t\geq 0\}$ for a particular stochastic reaction network. The network we choose will be the following \cite{LEFEVER1971267}.

\begin{example}
Consider two chemical species, $A$ and $B$, which undergo the following four reactions: 
\[ \emptyset \xrightarrow{a} A, \qquad A \xrightarrow{1} \emptyset, \qquad A \xrightarrow{b} B, \qquad B \xrightarrow{1} A\] 
according to the following propensity functions, 
\begin{equation}\label{brussprop}
     \alpha_1=a\Omega, \qquad \alpha_2=Y_1, \qquad \alpha_3=bY_1, \qquad \alpha_4 = Y_1^2Y_2/\Omega^2
\end{equation}
where $Y_1$ and $Y_2$ denote the populations of species $A$ and $B$ respectively and $a,b>0$ are the reaction rate constants. This canonical reaction network example is known as the Brusselator.
\end{example}

The notation used here is standard in the literature of reaction networks where $A+B \xrightarrow{k} C$ means the population of $A$ and the population of $B$ both decrease by one while the population of $C$ increases by one and the associated reaction rate constant for this reaction is $k$. Setting $Y(t) = (Y_1(t),Y_2(t))^\intercal$ and letting $X(t)\equiv Y(t)/\Omega$ where $\Omega$ denotes the system size, we see the Brusselator fits into our framework above with $d=2$ and $r=4$. We choose this system for its oscillatory properties; in the $\Omega\to\infty$ limit the deterministic process $\{ z(t) \ | \ t\geq0\}$ undergoes a Hopf bifurcation. That is, through small deviations of $b$ we deduce the Hopf bifurcation is supercritical with losses of stability occurring as the eigenvalues of the Jacobian matrix of \eqref{eq:RRE} cross the imaginary axis \cite{BROWN19951713}. The case $b<2$ gives rise to a stable fixed point, whereas $b>2$ gives rise to an unstable fixed point surrounded by a unique and stable limit cycle. This behaviour is displayed in Figure \ref{fig:3} - paths of $\{z(t) \ | \ t\geq0\}$ as well as sample paths of $\{X(t) \ | \ t\geq0\}$ for different values of $b$.

\begin{figure}
    \centering
    \includegraphics[width=0.8\textwidth]{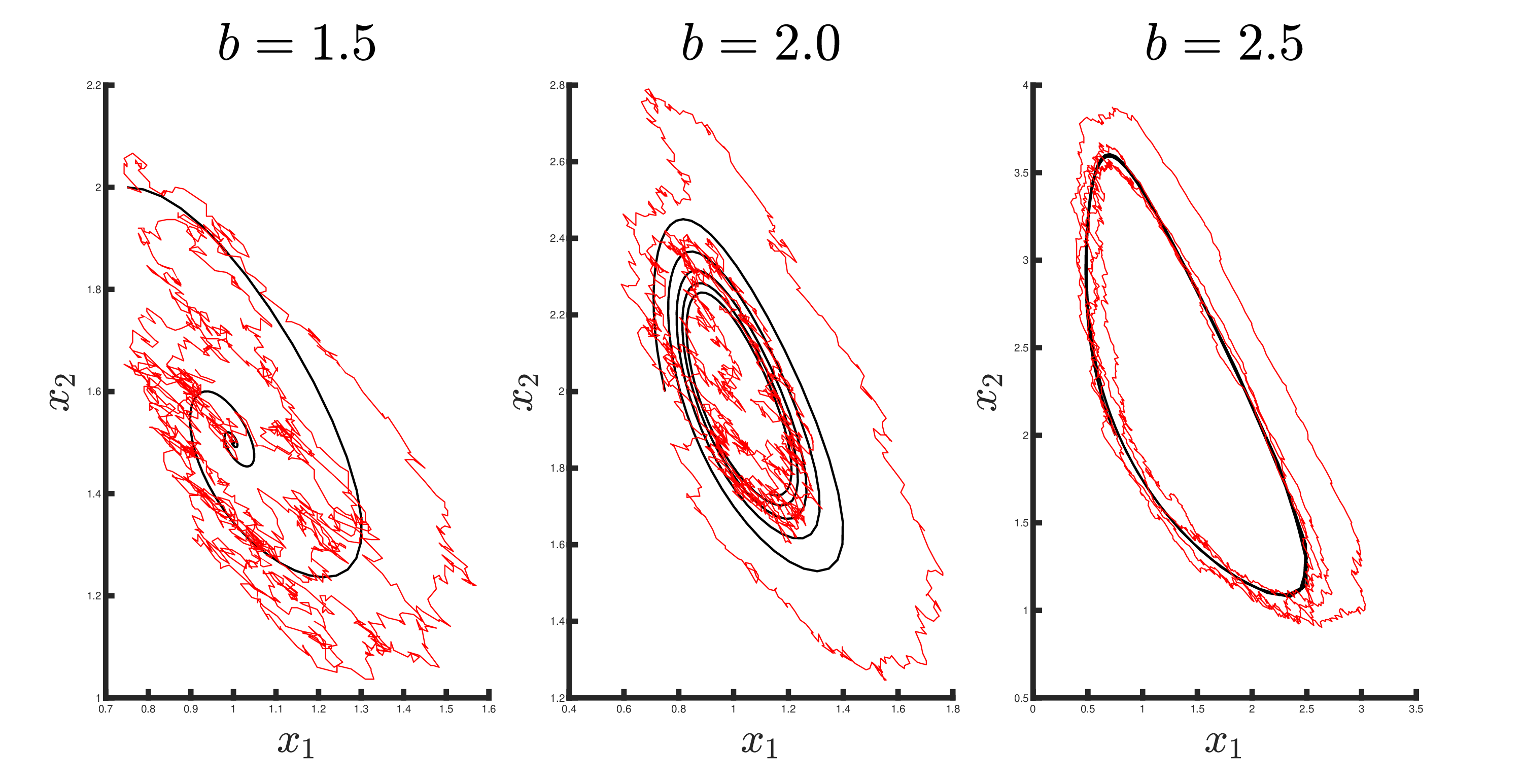}
    \caption{Phase portrait of sample paths of $\{z(t) \ | \ t\geq 0\}$ (black) and $\{X(t) \ | \ t\geq 0\}$ (red) for the Brusselator reaction network. Paths are shown for reaction rates $b=1.5, \ 2.0$ and  $2.5$ respectively with $z(0)=X(0)=(0.75,2)^\intercal$.}
    \label{fig:3}
\end{figure}

In order to match how a typical path of $\{X^\CLE(t) \ | \ t\geq 0\}$ is obtained and used in practical applications, we study its behaviour under the Euler-Maruyama scheme for \eqref{eq:CLE}. In other words we fix a time step $\tau>0$ and study $\{X^\CLE(t) \ | \ t\geq 0\}$ at discrete times $t=n\tau$ for $n\in\bbNz$, assuming it evolves according to the Markov chain $\{X^\CLE_n \ | \ n\in\bbNz\}$ where $X^\CLE_n\equiv X^\CLE(n\tau)$ and,
\begin{equation}\label{MarkovChain}
    X^\CLE_{n+1} = X^\CLE_n + \sum_{j=1}^r \nu_{j}\alpt_j(X^\CLE_n)\tau + \frac{1}{\sqrt{\Omega}} \sum_{j=1}^r \nu_{j} \sqrt{\alpt_j(X^\CLE_n)\tau} \ \mathcal{N}_j(0,1),
\end{equation}
where for each $j\in\{1,..,r\}$, the $\mathcal{N}_j(0,1)$ is an independent standard normal random variable. It turns out a suitable $\tau$ can be found in most stochastic reaction networks \cite{Gill2001,Gill2003,Gill77} . Under the definitions in the next section the evolution of $\{X^\CLE_n \ | \ n\in\bbNz\}$ and $\{X^\LNA(t) \ | \ t\geq 0\}$ can be described by a random dynamical system (RDS) and a two-parameter stochastic flow of diffeomorphisms respectively.

The paper will proceed as follows. For each approximation $\{X^\CLE_n \ | \ n\in\bbNz\}$ and $\{X^\LNA(t) \ | \ t\geq 0\}$ we will develop the framework we need to talk about some of the dynamical properties of paths for general stochastic reaction networks. We will then focus specifically on paths of the Brusselator reaction network and investigate the exponential rate of separation or contraction of paths infinitesimally close paths over both finite times (through finite time Lyapunov exponents) and in the long time limit (through Lyapunov exponents). This way we are able to discuss how sensitive the paths of $\{X^\CLE_n \ | \ n\in\bbNz\}$ and $\{X^\LNA(t) \ | \ t\geq 0\}$ are to their initial conditions. Paths exhibiting separation and contraction are deemed unstable and stable paths respectively. For $\{X^\CLE_n \ | \ n\in\bbNz\}$ we show discuss how the Lyapunov exponents vary over values of $b$ to try and identify what is known as a dynamical bifurcation in stochastic systems. Similar analysis can be seen for the deterministic Brusselator with a random bifurcation parameter $b$, both in the finite time frame \cite{engel2023noiseinduced} as well as in the long-time limit \cite{Arnold1999TheSB}. We believe we are the first to study these properties for the stochastic Brusselator whose noise is derived from the chemical physics of stochastic reaction networks. Similar to what was conjectured in \cite{Arnold1999TheSB}, we are able to prove a global synchronisation result as a consequence of these exponents. We then turn to its approximate process $\{X^\LNA(t) \ | \ t\geq 0\}$. We compare the separation and contraction of paths of both models over finite times allowing us to critique how well the LNA approximates the CLE from a new perspective of dynamics rather than statistics. Though both approximations demand similar conditions of validity and for each $t\geq 0$, both $X^\CLE(t)$ and $X^\LNA(t)$ boast convergence in distribution to $X(t)$ when $\Omega\to\infty$, we show that, for the Brusselator, typical paths of $\{X^\CLE(t) \ | \ t\geq 0\}$ and $\{X^\LNA(t) \ | \ t\geq 0\}$ exhibit very different finite time stability. For clarity, we first state the main results in Section 3 and accompany the results with proofs detailed in the sections thereafter in the order in which they are stated in Section 3.

\section{Statement of main results}

In part \ref{CLEsect} we will state the results concerning $\{X^\CLE_n \ | \ n\in\bbNz\}$ only. We follow this by part \ref{LNAsect} where we state the results concerning $\{X^\LNA(t) \ | \ t\geq 0\}$ and how its dynamical properties compare to that of $\{X^\CLE_n \ | \ n\in\bbNz\}$.

\subsection{The CLE}\label{CLEsect}

In order to single out individual realisations of the noise and consider their effect on the dynamics, we must formalise our model \eqref{MarkovChain}. Let $\ell^r$ denote the Lebesgue measure on $\bbR^r$ and consider the probability space $(\bbR^r,\cB(\bbR^r),\eta)$ where $\eta$ is the probability measure for the $MVN(0,I_r)$ distribution:
\[\eta(A) = \left(\frac{1}{2\pi}\right)^\frac{r}{2} \int_A e^{-\frac{1}{2}y^\intercal y} \  d\ell^r(y), \quad \text{for all } A\in\cB(\bbR^r). \]
We define the set,
\[ \Sigma := \{ \omega = (\omega_n)_{n\in\bbNz} \ | \ \omega_n\in\bbR^r \}, \]
along with the cylinder sets in $\Sigma$ of the form,
\[ \left[ A_0,A_1,...,A_k \right] := \{ (\omega_n)_{n\in\bbNz} \in \Sigma \ | \ \omega_i \in A_i, \ i\in\{0,...,k\} \}  \]
where $A_0,A_1,...,A_k\in\cB(\bbR^r)$. Let $\cF$ be the smallest Borel $\sigma$-algebra generated by the cylinder sets in $\Sigma$ and $\bbP$ be the unique measure on $\cF$, such that
\[ \bbP\left(  \left[ A_0,A_1,...,A_k \right] \right) = \prod_{i=0}^k \eta(A_i). \] 
One can think of each element $\omega\in\Sigma$ as a single noise realisation of the process \eqref{MarkovChain} as it evolves indefinitely. The $n^\textsuperscript{th}$ term of the sequence $\omega$ contains the $r$ realisations of the $r$ standard normal random variables generated at each step of the evolution law. Hence $\Sigma$ is the set of all possible realisations of the noise. On our new probability space $(\Sigma,\cF,\bbP)$, we define the shift map $\theta:\Sigma\to\Sigma$ and its iterates by,
\begin{equation}\label{shift}
    \theta\omega = \theta(\omega_0,\omega_1,...) = (\omega_1,\omega_2,...), \quad \theta^n := \underbrace{\theta\circ\theta\circ ... \circ\theta}_{n \text{ times}}.
\end{equation}
We impose $\theta^0 \equiv \mathrm{id}_{\Sigma}$. Now let $f:\Sigma\times\bbR^d \to \bbR^d, \ (\omega,x)\mapsto f_\omega(x)$ be such that,
\begin{equation}\label{f_omega}
    f_\omega(x) = x + \sum_{j=1}^r \nu_{j}\alpt_j(x)\tau + \frac{1}{\sqrt{\Omega}} \sum_{j=1}^r \nu_{j} \sqrt{\alpt_j(x)\tau} \ \omega_0^{(j)}  , \quad i \in \{1,...,d\}
\end{equation}
where we have used the same notation as in \eqref{MarkovChain} and $\omega_0 = (\omega_0^{(1)},...,\omega_0^{(r)})^\intercal$. It is worth noting that $f$ takes the whole sequence $\omega\in\Sigma$ as an input but only evaluates the first entry of $\omega$, namely $\omega_0$.

Finally, let $\varphi:\bbNz\times\Sigma\times\bbR^d\to\bbR^d$ where,
\begin{equation}\label{varphi}
    \varphi_{n+1}(\omega,x) = 
    \begin{cases}
        f_{\theta^n\omega}\circ \varphi_{n}(\omega,x) & \text{if } \varphi_{n}(\omega,x)\in M, \\
        \mathrm{id}_{\bbR^d}  & \text{if } \varphi_{n}(\omega,x)\in \bbR^d\setminus M. \\
    \end{cases}
\end{equation}
and $\varphi_0(\omega,\cdot) \equiv \mathrm{id}_{\bbR^d}$ for any $\omega\in\Sigma$. We define the map $\varphi$ in this way as we only concerned with dynamics within a compact subset of the positive quadrant, $M\subset\bbR^d_{>0}$. Indeed species concentrations cannot be negative and so we do not care for such paths which move outside of the positive quadrant. The compactness condition will become necessary in our analyisis later. We have decomposed our state space into $\bbR^d= M \sqcup \left( \bbR^d\setminus M\right)$ where $\sqcup$ denotes disjoint union. Now $\bbR^d\setminus M$ acts as an absorbing (or cemetery) state, killing off the process if one species population becomes zero or leaves $M$.
In the natural way, we define the stopping time for each $x\in\bbR^d$ and $\omega\in\Sigma$,
\begin{equation}\label{stoptime}
    \tn(\omega,x) = \inf\{n\in\bbNz: \varphi_n(\omega,x)\in\bbR^d\setminus M\}.
\end{equation}

Using this formulation, we can choose some starting point $x\in M$ and specific noise realisation $\omega = (\omega_0,\omega_1,...) \in \Sigma$ and see where the process is at time $n$ by applying the map $(\omega,x)\mapsto\varphi_n(\omega,x)$. By specifying a whole noise realisation $\omega\in\Sigma$ as the argument of these maps, we are able to analyse those realisations of the process \eqref{MarkovChain} which stay inside the desired region $M$. Providing $\tn(\omega,x)>n$, then $\varphi_{n}(\omega,x)$ coincides with $X^\CLE_n$ with $X^\CLE_0=x$  as required:

\begin{equation}\label{RDSeq}
    \varphi_{n+1}(\omega,x) = \varphi_{n}(\omega,x) + \sum_{j=1}^r \nu_{j}\alpt_j(\varphi_{n}(\omega,x))\tau 
        + \frac{1}{\sqrt{\Omega}} \sum_{j=1}^r \nu_{j} \sqrt{\alpt_j(\varphi_{n}(\omega,x))\tau} \ \omega^{(j)}_n  , \quad i \in \{1,...,d\}.
\end{equation}

\begin{theoremx}\label{rdsthm}
    Let $(\Sigma,(\cF_n)_{n\in \bbNz},\cF,\bbP,(\theta^n)_{n\in \bbNz})$ be the noise space and $\varphi$ be the map defined as above. Then, $(\theta,\varphi)$ is a \textit{memoryless RDS with absorption} on $\bbR^d$ over $(\theta^n)_{n\in \bbNz}$. That is,
    \begin{enumerate}[noitemsep,label=(\roman*)]
        \item $\varphi$ is $P(\bbNz)\otimes\cF\otimes\cB(\bbR^d)-\cB(\bbR^d))$-measurable,
        \item $\varphi_0(\omega,\cdot) = \mathrm{id}_{\bbR^d}$ for any $\omega\in\Sigma$,
        \item $\varphi_{n+m}(\omega,x) = \varphi_n(\theta^m\omega, \varphi_m(\omega,x))$ for all $n,m\in \bbNz$ and $x\in \bbR^d$ and any $\omega\in\Sigma$,       
    \end{enumerate}
    i.e. $(\theta,\varphi)$ is a RDS,
    \begin{enumerate}[noitemsep,label=(\alph*)]
        \item $(\Sigma,(\cF_n)_{n\in \bbNz},\cF,\bbP,(\theta^n)_{n\in \bbNz})$ is a memoryless noise space,
        \item $\varphi_n$ is $(\cF_n\otimes\cB(\bbR^d)-\cB(\bbR^d))$ measurable for every $n\in\bbNz$,
    \end{enumerate}
    i.e. $(\theta,\varphi)$ is memoryless,
    \begin{enumerate}[noitemsep,label=(\Alph*)]
        \item $\bbR^d=M\sqcup \{\partial\}$ where $M\subset \bbR^d$ and $\{\partial\}:=\bbR^d\setminus M$ denotes a cemetery state,
        \item For all $\omega\in\Sigma$ and $x\in \bbR^d, \ \varphi_m(\omega,x) = \partial \implies \varphi_n(\omega,x) = \partial$ for all $n\geq m$,
    \end{enumerate}
    i.e. $(\theta,\varphi)$ exhibits absorption.

    Moreover, for each $x\in M, \ \omega\in\Sigma$ and $n\in\bbNz$, $\varphi_n(\omega,x)$ denotes $X^\CLE_n$ if $X^\LNA_0=x, \ X^\CLE_n\in M$ and $\{X^\CLE_n \ | \ n\in\bbNz\}$ evolved according to the CLE process \eqref{MarkovChain} with noise realisation $\omega$.
\end{theoremx}

We now draw our focus to the Brusselator, where $d=2, \ r=4$ and the propensity functions take the form \eqref{brussprop}. In the notation above,
\begin{equation}\label{eq:fbruss}
     f_\omega(x) = F_0(x) + F(x)\omega_0,
\end{equation}
for each $\omega\in\Sigma$ and $x=(x_1,x_2)^\intercal\in M$ where,
\begin{equation}\label{eq:Fbrus}
    F_0(x) = \begin{pmatrix}
        x_1 + (a-(1+b)x_1 + x_1^2x_2)\tau \\
        x_2 + (bx_1-x_1^2x_2)\tau
    \end{pmatrix}, 
    \quad 
    F(x) = \sqrt{\frac{\tau}{\Omega}} \begin{pmatrix}
        \sqrt{a} & -\sqrt{x_1} & -\sqrt{bx_1} & x_1\sqrt{x_2} \\
        0 & 0 & \sqrt{bx_1} & -x_1\sqrt{x_2}
    \end{pmatrix}.
\end{equation}

From herein we let $(\theta,\varphi)$ be the RDS characterised in Theorem \ref{rdsthm} with its evolution determined by \eqref{varphi}, \eqref{eq:fbruss} and \eqref{eq:Fbrus} unless otherwise specified. We call $(\theta,\varphi)$ the Brusselator RDS.

With this explicit representation we show the Brusselator exhibits quasi-ergodicity in the sense of \cite{castro2024existence}. As an introduction to the subject, one can think of quasi-ergodicity as a sort of local version of ergodicity. We are able to use the fact that orbits which never escape a defined region, visit almost all areas of that region and so we are able to prove results about the average behaviour of such orbits through studying a 'typical' orbit. Like ergodicity, it opens our analysis up to a range of ergodic theorems, only applicable to those paths which stay in the desired region.

\begin{theoremx}\label{quasthm}
    The Brusselator RDS $(\theta,\varphi)$ admits a unique quasi-ergodic measure $\nu$ on $M$ and therefore exhibits quasi-ergodicty. That is, there exists a unique Borel measure $\nu$ on $M$ such that,
     \[ \lim_{n\to\infty} \bbE_x \left[ \frac{1}{n}\sum_{i=0}^{n-1} h\circ\varphi_i \ \middle| \ \tn>n \right] = \int_M h(y) \ d\nu(y)  \]
    holds for every $n\in\bbN$, $h\in L^\infty(M)$ and $\nu$-almost every $x\in M$.
   
    Moreover, given $\eta$, a Borel measure on M giving non zero probability to paths staying in $M$, for every $n\in\bbN$, there exists $K(\eta)>0$ and $\beta>0$, such that
    \[ \norm{\bbP_\eta(\varphi_n\in\cdot \ | \ \tn>n) -\mu}_{TV} \leq K(\eta)e^{-\beta n}, \]
    for all $n\in\bbN$. Here $\norm{\cdot}_{TV}$ denotes the total variation norm and $\bbP_\eta(\cdot) := \int_M (\bbP\otimes\delta_x)(\cdot) \ d\eta(x)$.
   
    Finally, $\nu$ corresponds to the stationary measure for paths which stay in $M$.
\end{theoremx}

A graphical representation of the density with respect to $\nu$ for different values of $b$ is shown in Figure \ref{fig:2}. The density function was calculated numerically by Ulam's method for approximating invariant measures \cite{LI1976,Bose2014}.

\begin{figure}
    \centering
    \includegraphics[width=\textwidth]{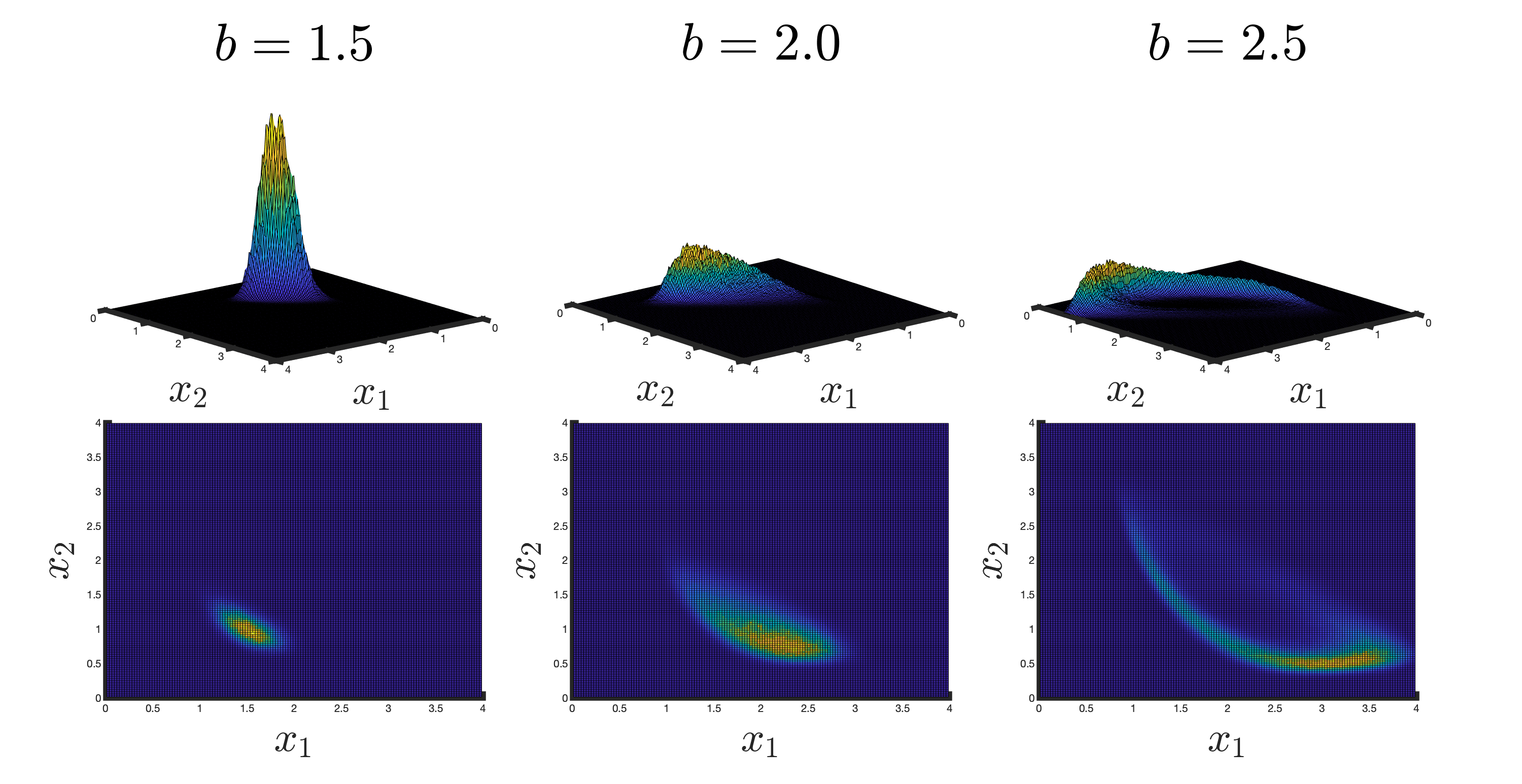}
    \caption{Quasi-ergodic density function of the Brusselator RDS $(\theta,\varphi)$ with $b = 1.5,2.0$ and $2.5$ respectively and system size $\Omega=1000$. The shape of the function is displayed for all $x=(x_1,x_2)^\intercal\in[0,4]\times[0,4]$. The lighter colours indicate higher probability.}
    \label{fig:2}
\end{figure}

We seek to appropriately describe the sensitivity of the long term behaviour of our RDS $(\theta,\varphi)$ to its initial conditions. More specifically, we want to quantify the rate of separation between infinitesimally close orbits conditioned on the fact that they haven't been absorbed. As we will find out, quasi-ergodicity will play a crucial role in quantifying this. The standard procedure to quantify the rate of separation of orbits of continuously differentiable random dynamical systems is through the \textit{Lyapunov exponents} \cite[Chapter 3]{arnold2013random} which we will now define. By continuously differentiable (class $\cC^1$) we mean, for all $n\in\bbNz$ and $\omega\in\Sigma$, the map $x\mapsto\varphi_n(\omega,x)$ is differentiable in $x$ with continuous derivative. We use $\norm{\cdot}$ to denote both the standard Euclidean norm on $\bbR^2$ and the corresponding matrix norm it induces.

\begin{definition}[Lyapunov exponents]\label{lyapexp}
    Let $(\theta,\varphi)$ be a continuously differentiable RDS. Let $\omega\in\Sigma, \ x \in M$ and $T_xM\subset\bbR^2$ denote the tangent space of $M$ at $x$. The Lyapunov exponent of $v\in T_xM$ is the defined as the limit,
    \[ \lambda_v(\omega,x) := \limsup_{n\to\infty}\frac{1}{n}\log \norm{\mathrm{D}_x\varphi_n(\omega,x)v} \]
    where $\mathrm{D}_x$ denotes the derivative operator with respect to $x$ and $v\mapsto\mathrm{D}_x\varphi_n(\omega,x)v$ is a linear map from $T_xM$ to $T_{\varphi_n(\omega,x)}M$.
\end{definition}
The linear map $v\mapsto\mathrm{D}_x\varphi_n(\omega,x)v$ defines the evolution of the tangent vectors of orbits of $(\theta,\varphi)$. Lyapunov exponents in stochastic bifurcation theory are the analogues to eigenvalues in determinstic bifurcation thoery; when trying to identify a bifurcation we look for a change in sign of Lyapunov exponents known as a D-bifurcation (or dynamical bifurcation). When $\lambda_v(\omega,x)<0$ then, along the direction of $v\in T_xM$, we have uniform mutual convergence of paths with intial condition infinitessimally close to $(\omega,x)\in\Sigma\times M$ at an exponential rate. Instead, if $\lambda_v(\omega,x)>0$ then we have uniform mutual divergence and hence a change in sign of $\lambda_v(\omega,x)>0$ results in a stark change in dynamics. In random dynamical systems theory this is known as a Dynamical bifurcation (or D-bifurcation).

However, for the Brusselator RDS $(\theta,\varphi)$ with orbits almost surely being absorbed by the cemetery state at some finite time, this limit defined on initial conditions $(\omega,x)\in\Sigma\times M$ is not very enlightening. This motivates us to define the notion of \textit{conditioned Lyapunov exponents}; the Lyapunov exponents of initial conditions $(\omega,x)$ corresponding to orbits which do not get absorbed \cite{Engel_2019}. For our purpose they provide the same information as ordinary Lyapunov exponents, except they only consider those paths which do not get absorbed. Hence we can extend the D-bifurcation theory to absorbing RDS. We prove a version of Oseledet's Multiplicative ergodic theorem.

\begin{theoremx}\label{osethm}
    Let $(\theta,\varphi)$ be the Brusselator RDS. There exists a measure $\bbQ_\nu$ on $\Sigma\times M$ which gives full measure to orbits of $(\varphi_n)_{n\in\bbN}$ which never get absorbed. In particular, $\bbQ_\nu(\Xi) = 1$ where $\Xi\in\cF\otimes\cB(M)$ has the form
    \[ \Xi := \{ (\omega,x)\in\Sigma\times M \ | \ \tn(\omega,x) = \infty \} = \bigcap_{n\in\bbN} \{ (\omega,x)\in\Sigma\times M \ | \ \tn(\omega,x) > n \}.  \]
     
     Moreover, there exists a set $\Delta\in \cF\otimes\cB(M)$, invariant under $(\theta,\varphi)$ and of $\bbQ_\nu$-full measure such that $\Delta\subset\Xi$ and for all $(\omega,x)\in\Delta$, there exists a vector line $E(\omega,x)\subseteq\bbR^2$ such that,
    \begin{equation*}
    \lambda_v(\omega,x) = \lim_{n\to\infty}\frac{1}{n}\log \norm{\mathrm{D}_x\varphi_n(\omega,x)v} = 
        \begin{cases}
            \lambda_1 & \text{if } 0\neq v\in \bbR^2\setminus E(\omega,x), \\
            \lambda_2 & \text{if } 0\neq v\in E(\omega,x).
        \end{cases}
    \end{equation*}    
    Here $\lambda_1 > \lambda_2>-\infty$ are the conditioned Lyapunov exponents and are constant. The subspace $E(\omega,x)$ is such that $\mathrm{D}_x\varphi_n(\omega,x)E(\omega,x) = E(\theta^n\omega,\varphi_n(\omega,x))$. 
    
    Finally for $\nu$-almost every $x\in M$,
    \[ \lim_{n\to\infty} \bbE_x \left[ \abs{\lambda_1- \frac{1}{n}\log\norm{\mathrm{D}_x\varphi_n(\omega,x)}} \ \middle| \ \tn>n \right] = 0 \quad \text{and} \quad \lim_{n\to\infty} \bbE_x \left[ \abs{\lambda_1+\lambda_2 - \frac{1}{n}\log\abs{\det\mathrm{D}_x\varphi_n(\omega,x)}} \ \middle| \ \tn>n \right] = 0. \]
    Here the matrix norm is induced by the Euclidean norm on $\bbR^2$.
\end{theoremx}

Importantly, for almost all initial conditions $(\omega,x)\in\Sigma\times M$ corresponding to non-absorbing orbits of the Brusselator RDS $(\theta,\varphi)$, the limit in Definition \ref{lyapexp} doesn't just exist beyond the $\limsup$ but in fact does not depend on the initial conditions $(\omega,x)$ and corresponds with invariant subspaces along the dynamics. Remarkably, the exponential rates of separation illustrated in Figure \ref{fig:4} for various values of $b$ apply to almost all paths of $\{X^\CLE_n \ | \ n\in\bbNz\}$.

\begin{figure}
    \centering
    \includegraphics[width=0.8\textwidth]{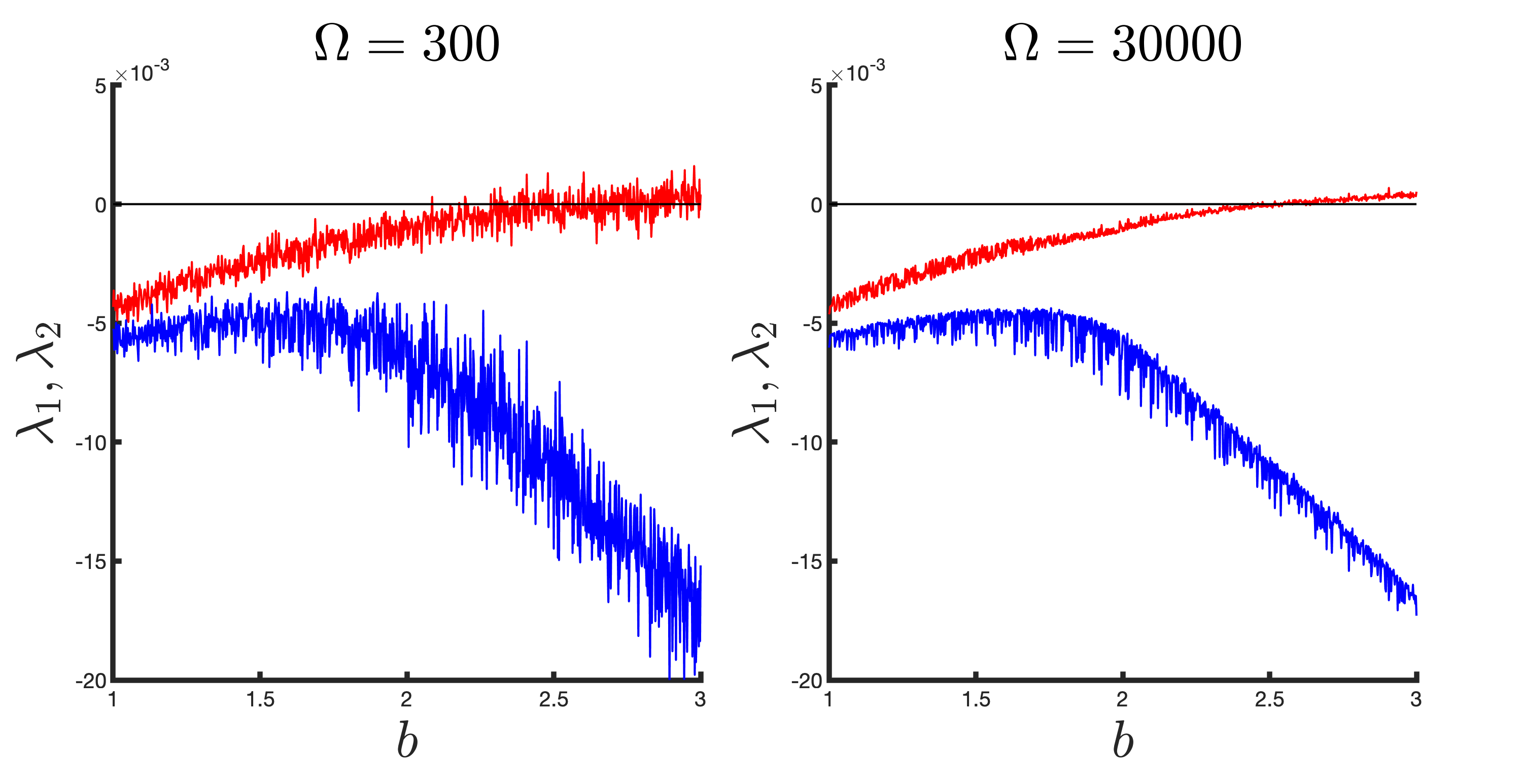}
    \caption{Conditioned Lyapunov exponents for the Brusselator RDS $(\theta,\varphi)$ for system sizes $\Omega=300,30000$. The red and blue lines represent the top Lyapunov exponent $\lambda_1$ and bottom Lyapunov exponent $\lambda_2$ respectively. }
    \label{fig:4}
\end{figure}

Figure \ref{fig:4} shows us the conditioned Lyapunov exponents are negative for almost all $b\in [1,3)$ except near the boundary $b=3$ where the top exponent $\lambda_1$ is null or positive. This implies for the the values of $b$ less than this boundary, the path of $\bbQ_\nu$-almost every initial condition in the phase space will be exponentially stable. That is, for such paths, there is a neighbourhood of its initial position such that the set of paths starting in this neighbourhood is uniformly mutually convergent at an exponential rate. If these paths were given the same noise realisation $\omega\in\Sigma$, then their paths would synchronise. We note the contrast to the deterministic Brusselator with random parameter $b$ \cite{Arnold1999TheSB} where the top exponent was negative for all parameter values studied.

Though some have succeeded \cite{f5037e4cc3ea4f31b0921943c03e2a27,Chemnitz2023, \cite{Arnold1999TheSB}}, proving the existence of non-negative exponents in stochastic Hopf bifurcations can be challenging \cite{Doan_2018}. Unfortunately we can only add to such examples where a mere conjecture of a non-negative $\lambda_1$ is all we can provide. In cases it can be shown that for $\lambda_1>0$, the path of $\bbQ_\nu$-almost every initial condition in the phase space will be exponentially unstable leading to local paths to diverge instead. In cases it can be shown that this leads to an interesting shaped global attractor of paths. Disappointingly we could not find numerical evidence of such for the Brusselator RDS $(\theta,\varphi)$. Instead we focus on the case where $\lambda_1<0$. Just as Newman \cite{NEWMAN_2018,newmanthesis} said, it is really only a local synchronisation property that is implied by the negativity of the Lyapunov exponents and so this begs the question: under what conditions under can we pass from such local properties to global stable synchronisation? We show that the Brusselator RDS $(\theta,\varphi)$ fulfils the necessary and sufficient conditions Newman found and highlight its consequences in the following theorem.

\begin{theoremx}\label{syncthm}
    If $\lambda_1<0$ then the Brusselator RDS $(\theta,\varphi)$ is (globally) stably synchronising with respect to $\bbQ_x$. That is,
    \begin{enumerate}[noitemsep,label=(\roman*)]
        \item For all $x,y\in M$,
        \[ \bbP^\bbQ_\nu(\{\omega \in \Sigma \ | \ \norm{\varphi_n(\omega,x)-\varphi_n(\omega,y)}\to 0 \text{ as } n\to\infty \}) = 1. \]
        \item For all $x\in M$,
              \[ \bbP^\bbQ_\nu(\{\omega \in \Sigma \ | \ \exists U\text{, a neighborhood of x, with }\diam(\varphi_n(\omega, U)) \to 0 \text{ as } n\to\infty \}) = 1 \]
    \end{enumerate}
    Here $\bbP^\bbQ_\nu$ is the marginal measure of $\bbQ_\nu$ on $\Sigma$.
        
    Hence, if $\lambda_1<0$ then there exists a global point attractor for $(\theta,\varphi)$. That is, there exists a random point $a:\Sigma\to M$ such that for $\bbP^\bbQ_\nu$-almost all $\omega\in\Sigma$,
        \[ \lim_{n\to\infty} \varphi_n(\theta^{-n}\omega, \cdot)_*\nu = \delta_{a(\omega)}, \]
    where $\delta_{a(\omega)}$ is the Dirac measure at $a(\omega)$ on $M$ and $\varphi_n(\theta^{-n}\omega, \cdot)_*\nu$ is the push forward measure of $\nu$ with respect to the map $x\mapsto\varphi_n(\theta^{-n}\omega, x).$
\end{theoremx}
It is worth noting that in order to consider negative time and define the inverse function of $\theta$ as we do above, we extend our $\Sigma$ to include doubly infinite sequence extending in both directions in the natural way. 

Numerically this phenomenon can be seen in Figure \ref{fig:5} where we have fixed a noise realisation $\omega\in\Sigma$ and consider a fine grid of initial points $x\in[0,5]\times[0,5]$ at times $-N$ for some large $N\in\bbN$. Then, we show that when the process $(\theta,\varphi)$ runs from this time, the state at time $n=0$ is the same for all initial starting points on this grid. In other words, the grid shrinks to the random point $a(\omega)\in M$ as $N\to\infty$. This is seen for $b=1.0$ and $2.0$. As expected, we cannot confirm whether or not this is observed in the case where $b=3.0$.

\begin{figure}[h]
    \centering
    \includegraphics[width=\textwidth]{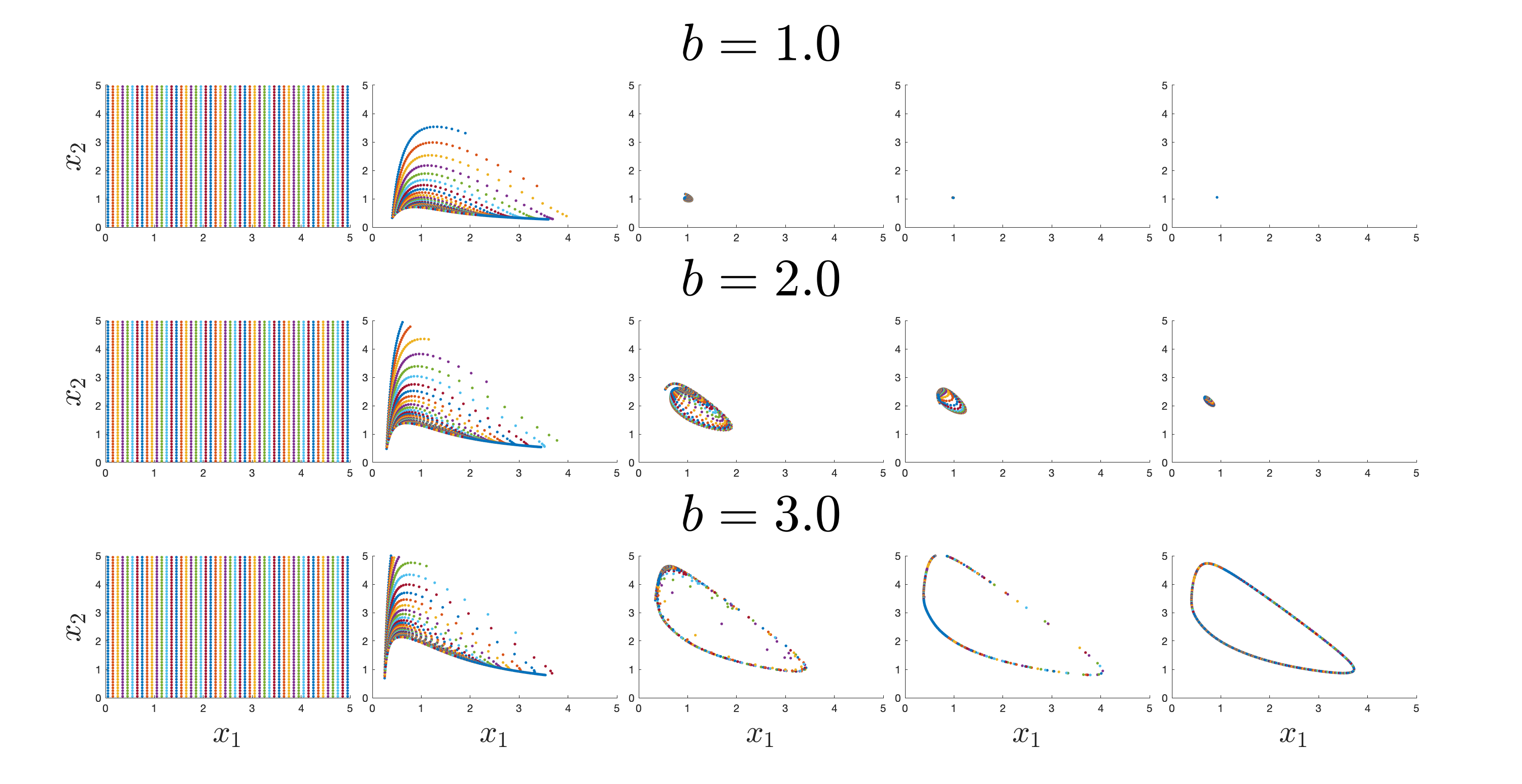}
    \caption{Pullback of the Brusselator RDS $(\theta,\varphi)$ for some randomly generated $\omega\in\Sigma^*$ with $b=1.0,2.0,3.0$ and $\Omega=1000$ from time $-N$ where $N=3000$. The first column of plots contain the fine grids of initial starting positions in $[0,5]\times[0,5]$ at time $n=-N$. The remaining columns contain plots of where each dot is mapped to under $(\theta,\varphi)$ at times $n = -1000, -500, -100$ and $n=0$ from left to right. }
    \label{fig:5}
\end{figure}

Although it remains an open question whether or not the Brusselator RDS $(\theta,\varphi)$ for $\{ X^\CLE_n \ | \ n\in\bbNz\}$ undergoes a D-bifurcation, we are able to prove a global synchronisation result for the case where $\lambda_1<0$, similar to the conjecture made in \cite{Arnold1999TheSB}.

\subsection{The LNA}\label{LNAsect}

We now turn our focus to $\{X^\LNA(t) \ | \ t\geq0\}$, the approximate process for $\{X^\CLE(t) \ | \ t\geq 0\}$. Again we begin by considering the general $d$-dimensional reaction network with its evolution defined by $r$ reactions.

Recall that for all $t\geq 0$ we are able to compute a distribution for $X^\LNA(t) = z(t) + \Omega^{-1/2}\xi(t)$ by computing a distribution for $\xi(t)$. This is obtained by solving \eqref{eq:xiSDE} (or equivalently the system \eqref{eq:RRE}, \eqref{eq:Code} and \eqref{eq:Vode}). As in the previous section, we need a formal model for the system's noise and how individual realisations of the noise effect the family of distributions for $\{X^\LNA(t) \ | \ t\geq0\}$. Fortunately, since the randomness of the model is inherited from $\{\xi(t) \ | \ t\geq0\}$, a solution of the SDE \eqref{eqn:xistate}, then our noise space is the canonical Wiener space of continuous functions $\omega^*:[0,\infty)\to\bbR^r$. We will denote this space by $(\Sigma^*,\cF^*,\bbP^*)$ where $\bbP^*$ is the Wiener measure. In a similar way to previously, for every $t\geq 0$, We want to define a map which takes an initial point $x\in\bbR^d$ and noise realisation $\omega^*\in\Sigma^*$ and maps it to where it would be on $\bbR^d$ at time $t$ if it evolved according the LNA under that noise realisation. However, unlike the CLE, the LNA depends on another input; the solution $\{z(t) \ | \ t\geq 0, \ z(0) = z_0\}$ to the RRE \eqref{eq:RRE} for some initial value $z_0\in\bbR^d$. It is clearly seen from the explicit forms of \eqref{eq:Code} and \eqref{eq:Vode} that the matrices $C$ and $V$ which are required to compute the perturbation vectors $\xi$ depend on $\{z(t) \ | \ t\geq0, \ z(0) = z_0\}$. This motivates us to define the following map. 

For every solution $\{z(t) \ | \ t\geq 0, \ z(0) = z_0\}$ of the RRE \eqref{eq:RRE}, define the family of continuous mappings from $\bbR^d$ to itself, $\{\Psi_{s,t}^{z_0} (\omega^*) \ | \ \omega^*\in\Sigma^*, \ s,t\geq 0\}$ where if $x\in\bbR^d$ then,
\begin{equation}\label{eq:LNAflow}
    \Psi_{s,t}^{z_0} (\omega^*)x = z(t) + C(s,t)(x-z(s)) + \Omega^{-1/2}\sum_{j=1}^r \int_s^t C(u,t)\nu_j\sqrt{\alpt_j(z(u))} \ d\omega_j^*(u)
\end{equation}
where $\omega^*(t) = (\omega_1^*(t),...,\omega_r^*(t))^\intercal$. Now if our systems state at time $s$ was $x\in\bbR^d$, then under dynamics of the LNA with RRE solution $\{z(t) \ | \ t\geq 0, \ z(0) = z_0\}$ and noise realisation $\omega^*\in\Sigma^*$, the system state at time $t$ is given by  $\Psi_{s,t}^{z_0} (\omega^*)x$. In this notation, $\Psi_{0,t}^{z_0} (\omega^*)x$ coincides with $X^\LNA(t)$ where $X^\LNA(0)=x$ and the RRE solution $\{z(t) \ | \ t\geq 0, \ z(0) = z_0\}$ is used.

Due to its non-autonomous nature, $\Psi_{s,t}^{z_0} (\omega^*)$ does not satisfy the cocycle property of RDS and hence the techniques used in the previous section cannot be applied here. Instead we classify this map in the following theorem.

\begin{theoremx}\label{stochdiffthm}
    Let $\{z(t) \ | \ t\geq 0, \ z(0) = z_0\}$ be a solution of the RRE \eqref{eq:RRE} for some $z_0\in\bbR^2$. Then the family of continuous mappings $\{\Psi_{s,t}^{z_0} (\omega^*) \ | \ \omega^*\in\Sigma^*, \ s,t\geq 0\}$ defined in \eqref{eq:LNAflow} forms a two-parameter stochastic flow of diffeomorphisms on $\bbR^d$. That is, for $\bbP^*$-almost all $\omega^*\in\Sigma^*$,
        \begin{enumerate}[noitemsep,label=(\roman*)]
            \item For all $s,t,u\in[0,\infty)$, $\Psi_{s,t}^{z_0} (\omega^*)=\Psi_{u,t}^{z_0} (\omega^*)\circ\Psi_{s,u}^{z_0} (\omega^*)$,
            \item For all $s\in[0,\infty)$, $\Psi_{s,s}^{z_0} (\omega^*)=\mathrm{id}_{\bbR^d}$,
            \item For all $s,t\in[0,\infty)$, $\Psi_{s,t}^{z_0} (\omega^*):\bbR^d\to\bbR^d$ is an onto homeomorphism,
            \item For all $s,t\in[0,\infty)$, $\Psi_{s,t}^{z_0} (\omega^*)x$ is differentiable with respect to $x\in\bbR^d$ and its derivative is continuous in $s,t$ and $x$.
        \end{enumerate}
    Moreover, for each $x\in\bbR^d, \ \omega^*\in\Sigma^*$ and $s,t\geq 0$, $\Psi_{s,t}^{z_0} (\omega^*)x$ denotes $X^\LNA(t)$ if $X^\LNA(s)=x$ and $\{X^\LNA(t) \ | \ t\geq0\}$ evolved according to the LNA \eqref{LNAansatz} with RRE solution $\{z(t) \ | \ t\geq 0, \ z(0) = z_0\}$ and noise realisation $\omega^*$.
\end{theoremx}

Using the theory of two-parameter stochastic flows of diffeomorphisms \cite{Kunita_1990}, one could perform a similar analysis to the last section, on the long-tine dynamics of the LNA. However, as mentioned before, \cite{Scott2006,Tomita1974,Boland2008,Ito2010,Minas2017} showed us that due to its nonlinearity in propensity functions, the LNA with finite $\Omega$ will fail to capture the true dynamics of the Brusselator \eqref{eq:rtc2} as $t\to\infty$ and so we question the need for such analysis. Instead, we focus on the finite time stability of the LNA and show that despite its aforementioned accuracy over finite times, paths of $\{X^\LNA(t) \ | \ t\geq0\}$ and $\{X^\CLE_n \ | \ n\in\bbNz\}$ exhibit very different finite time stabilities.

To this end, the object we study is an extension of the maximal finite time Lyapunov exponent to our stochastic flows $\{\Psi_{s,t}^{z_0} (\omega^*) \ | \ \omega^*\in\Sigma^*, \ s,t\geq 0\}$. 

\begin{definition}[LNA's maximal finite time Lyapunov exponent]\label{ftlyapexp}
    Let $\omega^*\in\Sigma^*$, $x\in\bbR^d$, $T>0$ and $\{z(t) \ | \ t\geq 0, \ z(0) = z_0\}$ be a solution of the RRE \eqref{eq:RRE}.
    
    Suppose that we have two paths which evolve under the LNA with the same RRE solution $\{z(t) \ | \ t\geq 0, \ z(0) = z_0\}$ and noise realisation $\omega^*$, one starting from $x$ and the other infinitesimally close to $x$. Let the maximal average rate of separation of these paths over the time interval $[0,T]$ be denoted by $\lambda^{z_0}_T(\omega^*,x)$. Note that this comparison is taken over all infinitesimal displacements from $x$. 
    
    The random scalar quantity $\lambda^{z_0}_T(\omega^*,x)$ is the LNA's maximal finite time Lyapunov exponent with respect to $\omega^*, \ x, \ T$ and $\{z(t) \ | \ t\geq 0, \ z(0) = z_0\}$. 
\end{definition}

We explore the explicit form of this quantity in the following theorem.

\begin{theoremx}\label{mftlethm}
    For every stochastic reaction network modelled by the LNA, the maximal finite time Lyapunov exponent $\lambda^{z_0}_T(\omega^*,x)$ is given by,
        \[ \lambda^{z_0}_T = \lambda^{z_0}_T(\omega^*,x) =  \frac{1}{T}\log\norm{C(0,T)}. \]
    Here the matrix norm $\norm{C(0,T)}$ is induced by the Euclidean norm on $\bbR^d$, or equivalently, the largest singular value of $C(0,T)$. 

    In particular, the LNA's maximal finite time Lyapunov exponent is independent of the noise realisation $\omega^*$ and initial point $x$.
\end{theoremx}

An important remark is that due to the noise independence of the LNA's maximal finite time Lyapunov exponent, we do not need to isolate our analysis to a subset of $\Sigma^*$ for which paths of $\{X^\LNA(t) \ | \ t\geq 0\}$ stay in the positive quadrant indefinitely; the exponent will have the same value for these paths.

If one desired, the LNA's maximal Lyapunov exponent $\lambda^{z_0}$ could be obtained by taking the limit,
\[ \lambda^{z_0} = \limsup_{T\to\infty}\frac{1}{T}\log\norm{C(0,T)}. \]

\begin{figure}
    \centering
    \includegraphics[width=\textwidth]{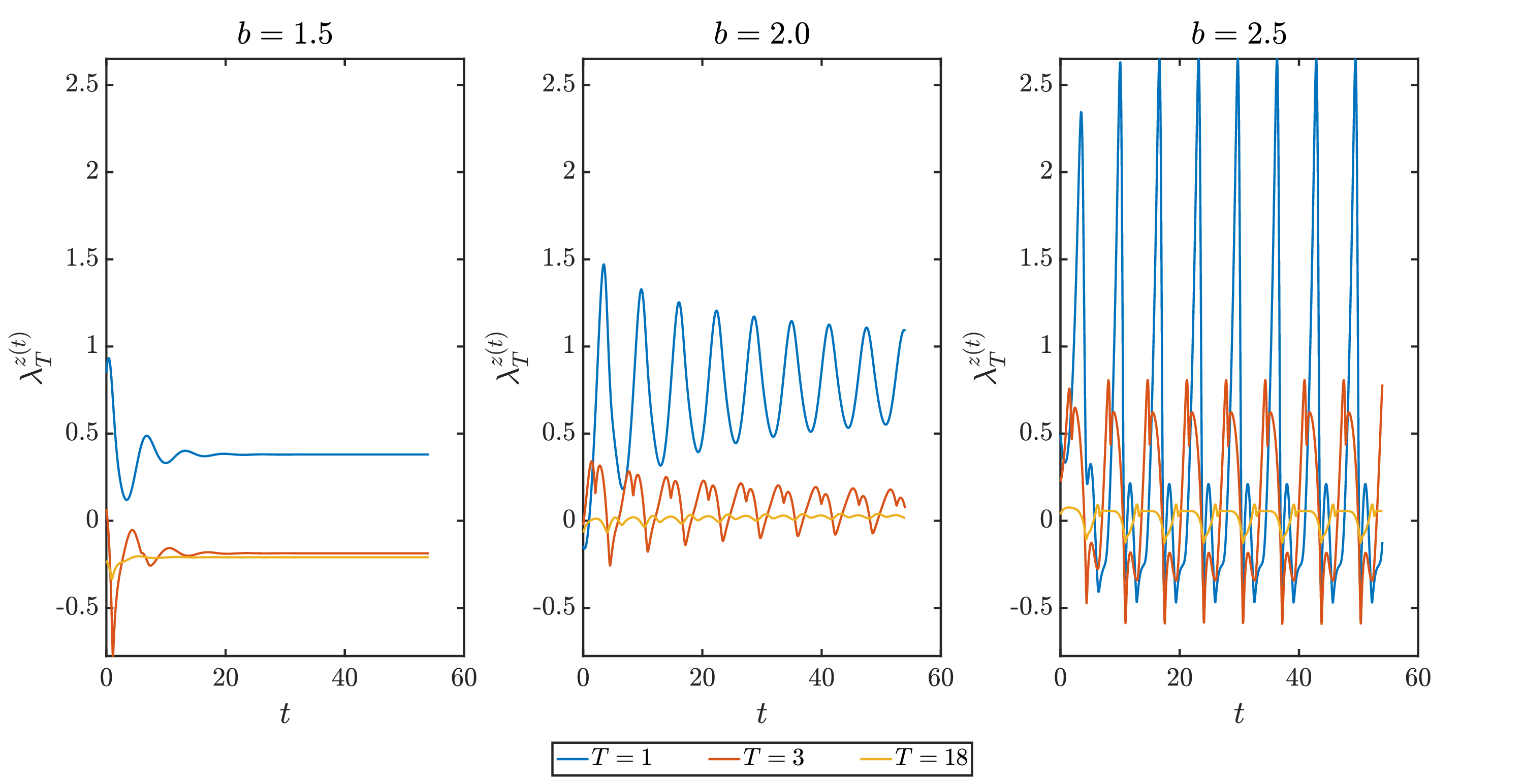}
    \caption{The LNA's maximal Lyapunov exponent $\lambda^{z_0}_T$ for $z_0\in\{z(t) \ | \ t\geq0, \ z(0) = (0.75,2)^\intercal\}$ computed for the Brusselator reaction network with $t\in[0,54]$. The first, second and third plot correspond to $b=1.5,2.0$ and $3$ respectively. Within each plot, the blue, orange and yellow lines correspond to $T=1, 3$ and $18$ respectively.}
    \label{fig:6}
\end{figure}

We now apply this theorem to study the Brusselator reaction network. In Figure \ref{fig:6}, we compute $\lambda^{z_0}_T$ as $z_0$ varies for roughly a sixth of an oscillation $(T=1)$, half an oscillation $(T=3)$ and three oscillations $(T=18)$. In particular we study $\lambda^{z(t)}_T$ as $t$ varies, where $\{z(t) \ | \ t\geq0, \ z(0) = (0.75,2)^\intercal\}$ is a solution to the RRE \eqref{eq:RRE} starting from $(0.75,2)^\intercal$. That is, how does the maximal finite time Lyapunov exponent vary if the LNA is simulated at various points along the RRE solution starting from $(0.75,2)^\intercal$. To this end, we are able to investigate the finite time behaviour of not just the $\{X^\LNA(t) \ t\geq 0\}$ starting at any point on $\bbR^2$ with $z_0 = (0.75,2)^\intercal$ but also the finite time behaviour of any method similar to Minas and Rand's model, discussed earlier, where the LNA is used many times from various points along the RRE solution. We repeat this for bifurcation parameter values $b=1.5,2.0$ and $2.5$. 

For $b < 2$ and $T\leq 1$, $\lambda^{z_0}_T>0$ for all $z_0$ tracing the solution $\{z(t) \ | \ t\geq0 \}$ indicating separation of all paths starting in the Euclidean plane for such $z_0$. This behaviour is the same for $b=2$ and almost all $z_0$ along this RRE solution. For $b=2.5$ we have separation for all $z_0$, except for a short interval in the limit cycle solution of $\{z(t) \ | \ t\geq0 \}$. As $T$ increases the rate of separation decreases. For $b\leq 2$ we can see numerically that for $T$ large enough, $\lambda^{z_0}_T$ will cross zero for all $z_0\in \{z(t) \ | \ t\geq0, \ z(0) = (0.75,2)^\intercal\}$ resulting in eventual contraction of paths and hence synchronisation in the local sense. 

To aid comparison, we introduce the CLE's maximal finite time Lyapunov exponent.
\begin{definition}[CLE's maximal finite time Lyapunov exponent]\label{ftlyapexp2}
    Let $\omega\in\Sigma$, $x\in\bbR^d$ and $T \in \{ t\in[0,\infty) \ | \ t = n\tau, \ n\in\bbNz \}$. The CLE's maximal finite time Lyapunov exponent with respect to $\omega, \ x$ and $T$ is given by,
    \begin{equation*}
        \lambda_T(\omega,x) := \sup_{v\in T_x\bbR^2} \frac{1}{T}\log \norm{\mathrm{D}_x\varphi_{T/\tau}(\omega,x)v},
    \end{equation*}
    where $T_x\bbR^2\subset\bbR^2$ denotes the tangent space of $\bbR^2$ at $x$.
\end{definition}

Note that this is equivalent to defining $\lambda_T(\omega,x)$ as the maximal average rate of separation over the time interval $[0,T]$ of two paths which evolve under the CLE with noise realisation $\omega$ and one starting from $x$ and the other infinitesimally close to $x$. Also, by taking the limit $T\to\infty$ we return to the maximal Lyapunov exponent,
\[ \lambda_1 = \lim_{T\to\infty}\lambda_T(\omega,x) = \lim_{T\to\infty}\sup_{v\in T_x\bbR^2} \frac{1}{T}\log \norm{\mathrm{D}_x\varphi_{T/\tau}(\omega,x)v}. \]

\begin{figure}
    \centering
    \includegraphics[width=\textwidth]{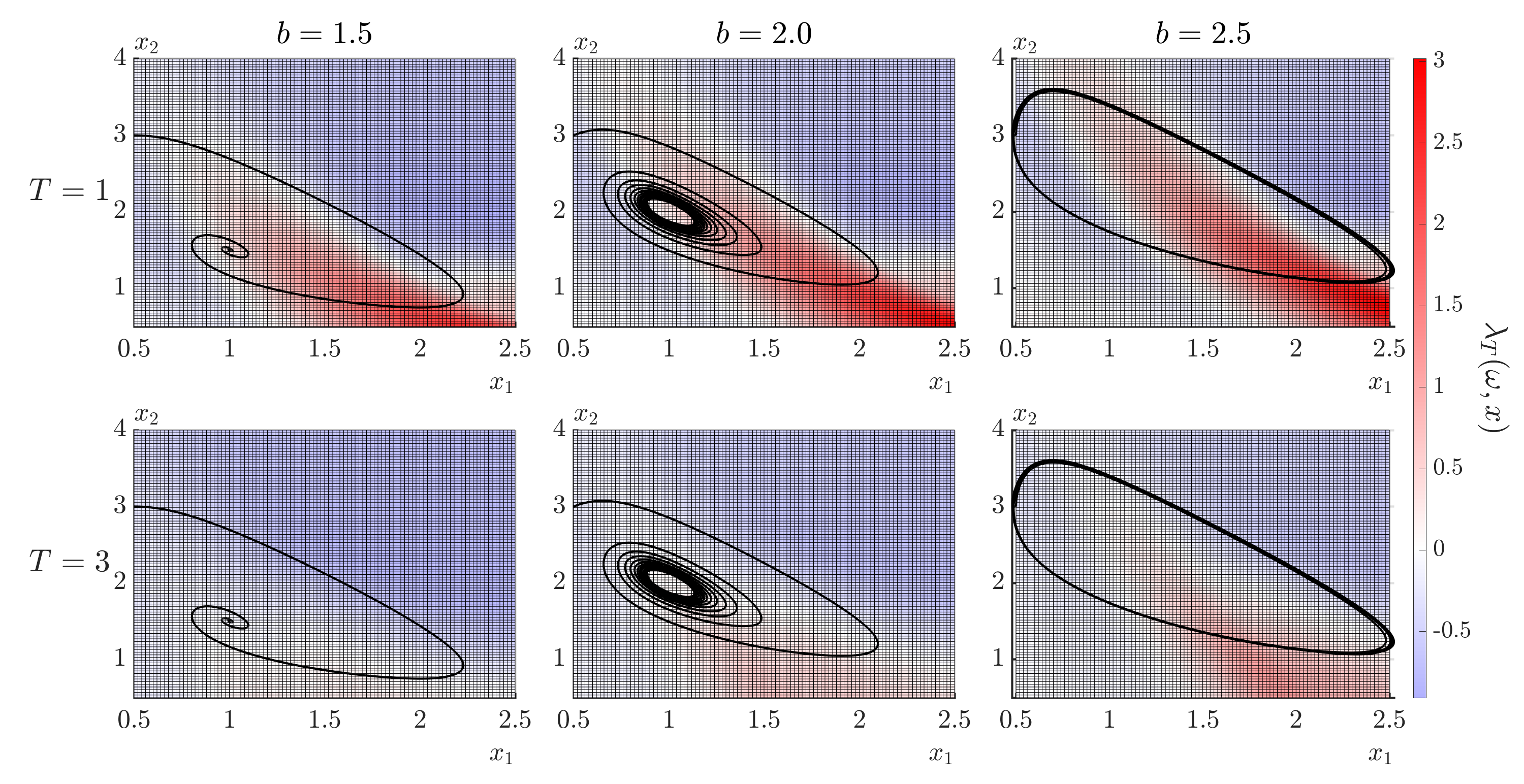}
    \caption{Landscapes of the CLE's average maximal finite time Lyapunov exponent $\lambda_T(\omega,x)$ for the Brusselator reaction network over $10^4$ realisations of $\omega\in\Sigma$ computed for $x\in[0.5,2.5]\times[0.5,4]$. The first and second row correspond to $T=1$ and $3$ respectively and the first, second and third columns correspond to $b=1.5,2.0$ and $3$ respectively. Blue regions indicate $\lambda_T(\omega,x)<0$, white regions indicate $\lambda_T(\omega,x)=0$ and red regions indicate $\lambda_T(\omega,x)>0$. In each plot, the black line denotes the RRE solution $\{z(t) \ | \ t\geq0, \ z(0) = (0.5,3)^\intercal\}$. }
    \label{fig:7}
\end{figure}

In Figure \ref{fig:7} we depict a landscape of the CLE's maximal finite time Lyapunov exponent for the Brusselator reaction network. For the same values of $b$ investigated in Figure \ref{fig:6} and for $T=1$ and $3$. Over $10^4$ noise realisations, we compute the average $\lambda_T(\omega,x)$ over a mesh of $100\times 100$ points in $\bbR^2$, close to the RRE solution $\{z(t) \ | \ t\geq0, \ z(0)=(0.5,3)^\intercal \}$.

Immediately we see very different finite time instabilities for paths of the CLE as we do with paths of the LNA. That is, contrary to the LNA's maximal Lyapunov exponents, $\lambda_T(\omega,x)$ is dependent on the initial point $x\in\bbR^2$, with some parts of the state space exhibiting separation while most of the state space exhibits contraction of local paths. For small $T$ this phenomenon is more prominent while, on average, for $T\geq 3$, $\lambda_T(\omega,x)<0$ for almost all $x\in\bbR^2$ indicating contraction of almost all local paths. We highlight two key differences in the finite time dynamics of the CLE and LNA. First, for small $T<3$, almost all paths of the LNA exhibit separation from their neighbouring paths whereas there is a mixture of separation and contraction of paths of the CLE depending on where they are in the state space. This contrast is even more pronounced with smaller $T$. Secondly, for $b\geq 2$ the rate at which regions with positive exponents decrease with $T$ for the LNA appears longer than the average time frame for such regions of the CLE to do the same.

We conclude that, despite claims that the LNA will approximate oscillatory dynamics the CLE well for short times, the LNA fails to capture some dynamical properties of oscillatory paths of the CLE. We have also shown that this is the case for all types of models seen in \cite{Minas2017,Minas2024} where repeated use of the LNA is used. With such models their accuracy depends on only using the LNA over very short times and so it is alarming to see larger discrepancies in the LNAs and CLEs finite time maximal Lyapunov exponents as the time interval of interest decreases. 

\section{The CLE and quasi-ergodicity}

Using the maps defined in \eqref{shift}, \eqref{f_omega} and \eqref{varphi} and the noise space outlined at the start of Section \ref{CLEsect}, we leave the list of conditions required to prove Theorem \ref{rdsthm} to be readily verified by the reader. 

We now focus on our example, the Brusselator, whose RDS $(\theta,\varphi)$ evolves according to \eqref{varphi},\eqref{eq:fbruss} and \eqref{eq:Fbrus}. In order to prove the existence of a quasi-ergodic measure we follow the theory outlined by Castro in \cite{castro2024existence,CASTRO_2023}. First, we recast our perspective to time-homogeneous Markov chains.

\begin{corollary}\label{Markovcor}
    The Brusselator RDS $(\theta,\varphi)$ induces the\textit{ time-homogeneous absorbed Markov chain}
    \[ \varphi = \left(  \Sigma\times\bbR^2, (\cF_n\otimes\cB(\bbR^2))_{n\in \bbNz}, (\varphi_n)_{n\in\bbNz}, (\cP^n)_{n\in \bbN}, (\bbP_x)_{x\in \bbR^2}    \right) \]
    on the state space $(\bbR^2,\cB(\bbR^2))$ in the sense of \cite[Definition III.1.1]{rogerswilliams1} where $\bbP_x := \bbP\otimes\delta_x$ and $\cP^n(x,A) = \bbP(\{\omega\in\Sigma \ | \ \varphi_n(\omega,x)\in A\})$.
\end{corollary} 
\begin{proof}
    See \cite[Chapter 2.5]{newmanthesis}.
\end{proof}

Here, $(\bbP_x)_{x\in \bbR^2}$ is a family of probability measures on $(\Sigma,(\cF_n)_{n\in \bbNz},\cF)$ where $\bbP_x$ is the probability measure for the process starting at $x\in \bbR^2$ and $(\cP^n)_{n\in \bbN}$ is a transition function on the state space $(\bbR^2,\cB(\bbR^2))$; a family of kernels $\cP^n:\bbR^2\times\cB(\bbR^2)\to [0,1]$ such that for $n,m\in\bbN, \ x\in \bbR^2$ and $A\in\cB(\bbR^2)$, $\cP^n(x,\cdot)$ is a probability measure on $\cB(\bbR^2)$, $\cP^n(\cdot,A)$ is a $\cB(\bbR^2)$-measurable function and $\cP^n$ satisfies the Chapman-Kolmogorov equation,
        \[\cP^{n+m}(x,A) = \int_{\bbR^2} \cP^n(y,A) \cP^m(x,dy).\]
We will refer to the time-homogeneous absorbed Markov chain $\varphi$ as the Brusselator Markov chain. It turns out the existence of the quasi-ergodic measure is deeply rooted in the properties of the transition kernel $\cP:=\cP^1$. Throughout the rest of this section, $\norm{\cdot}$ represents the standard Euclidean norm on $\bbR^2$.

\begin{lemma}
    Let $x\in M$ and $A\in\cB(\bbR^2)$. The Brusselator Markov chain $\varphi$ has
    \begin{equation*}
        \cP(x,A) = \int_A \kappa(x,y) \ d\ell^2(y)
    \end{equation*}
    where
    \begin{equation*}
        \kappa(x,y) := \frac{\Omega}{2\pi\tau\sqrt{\gamma(x)}}  \exp\left[-\frac{(y-F_0(x))^\intercal  \Gamma(x)^{-1}(y-F_0(x))}{2}\right]
    \end{equation*}
    and for $x = (x_1,x_2)^\intercal$,
    \begin{equation*}
        \Gamma(x) := \frac{\tau}{\Omega}  \begin{pmatrix}
         a+(1+b)x_1+x_1^2x_2 & -bx_1-x_1^2x_2 \\
         -bx_1-x_1^2x_2 & bx_1+x_1^2x_2
        \end{pmatrix},
        \quad \gamma(x) := \left(\frac{\Omega}{\tau}\right)^2 \det\Gamma(x) 
    \end{equation*}
    Note that $\cP$ is well defined for $x\in M \subset \bbR^2_{>0}$ as the roots of $\gamma$ are all non-positive.
\end{lemma}

\begin{proof}
    Let $x\in M$ and $A\in\cB(\bbR^2)$. Let $q_{X_n} \equiv F(X_n)\omega_n$ so that,
    \[  X_{n+1} = F_0(X_n) + F(X_n)\omega_n = F_0(X_n) + q_{X_n} \]
    We use the linear combination property of the multivariate normal distribution to deduce, 
    \[ q_{X_n} \sim MVN(0, \Gamma(X_n)) \]
    where
    \[ \Gamma(x) = F(x)F(x)^\intercal  = \frac{\tau}{\Omega}  \begin{pmatrix}
         a+(1+b)x_1+x_1^2x_2 & -bx_1-x_1^2x_2 \\
         -bx_1-x_1^2x_2 & bx_1+x_1^2x_2
        \end{pmatrix}. \]
    with 
    \[  \det\Gamma(x) = \left(\frac{\tau}{\Omega}\right)^2 x_1(x_1+a)(x_1x_2+b). \]
    Writing $\gamma(x) = x_1(x_1+a)(x_1x_2+b)$ we see that $a,b>0$ implies all the roots of $\gamma$ are non-positive, so $\det\Gamma(x)\neq0$ for any $x\in M \subset \bbR^2_{>0}$. Hence $\Gamma(x)$ is invertible for all $x\in M$.

    Let $\eta_x$ denote the probability measure for the $MVN(0,\Gamma(x))$ distribution then by definition,
    \begin{align*}
        \cP(x,A) &= \int_{\bbR^2} \chi_A(f_0(x) + q_x) \ d\eta_x(q_x) \\
        &= \frac{\Omega}{2\pi\tau\sqrt{\gamma(x)}} \int_{\bbR^2} \chi_A(f_0(x) + q_x) \exp\left[-\frac{q_x^\intercal \Gamma(x)^{-1}q_x}{2}\right] \ d\ell^2(q_x) \\
         &= \frac{\Omega}{2\pi\tau\sqrt{\gamma(x)}} \int_{\bbR^2} \chi_A(y) \exp\left[-\frac{(y-f_0(x))^\intercal \Gamma(x)^{-1}(y-f_0(x))}{2}\right] \ d\ell^2(y) \\
         &=  \frac{\Omega}{2\pi\tau\sqrt{\gamma(x)}} \int_A \exp\left[-\frac{(y-f_0(x))^\intercal \Gamma(x)^{-1}(y-f_0(x))}{2}\right] \ d\ell^2(y).
    \end{align*}
    The third equality comes from setting $y = f_0(x) + q_x$. In this form we may read off $\kappa$.
\end{proof}

\begin{proposition}\label{propPprop}
    Let $\varphi$ be the Brusselator Markov chain. The transition kernel $\cP$ has the following properties:
          \begin{enumerate}[label=(\roman*)]
            \item For all $x\in M$ and every $\varepsilon>0$, there exists a $\delta>0$ such that,
             \[ \norm{x-z} < \delta \ \implies \ \abs{ \kappa(x,y) -   \kappa(z,y) } < \varepsilon. \]
            \item For all $x\in M$ and every open set $A\in\cB(M)$, $\cP(x,A) >0$.
            \end{enumerate}
\end{proposition}
\begin{proof}
    For (i), let $\varepsilon>0$. Fixing $y\in M$, note that $\kappa(\cdot,y):M\to\bbR$,
    \[ \kappa(x,y) = \frac{\Omega}{2\pi\tau\sqrt{\gamma(x)}}  \exp\left[-\frac{(y-F_0(x))^\intercal  \Gamma(x)^{-1}(y-F_0(x))}{2}\right]. \] 
    is a continuous function on a compact space and so is uniformly continuous on $M$. Hence, the implication holds.
    
    For (ii), let $x\in M$, $A\in\cB(M)$ be open and $y\in A$. Since $A$ is open, there exists a $\delta>0$ such that, $B(y,\delta)\subset A$ where $B(y,\delta)$ denotes the open ball of radius $\delta$ centered at $y$. Also since $\kappa(x,y)>0$ for all $x,y\in M$ then there exists a $K>0$ such that $\kappa(x,y)>K$ for all $x\in B(y,\delta)$, then
        \[ \cP(x,A) = \int_A \kappa(x,y) \ d\ell^2(y) \geq \int_{B(y_0,\delta)} K \ d\ell^2(y)= K\ell^2(B(y_0,\delta)) > 0. \] 
\end{proof}

As we wish to investigate the asymptotic dynamics of orbits that remain inside $M$ and do not become absorbed, it is natural to seek the existence of stationary measures in the classical sense. However, due to the loss of mass by absorption at the boundary, this is impossible in our case. Hence we seek what is called a quasi-stationary measure which preserves mass along the process conditioned on survival.

In order to define a quasi-stationary measure, we introduce the notation, 

\begin{notation}
    Let $(\bbP_x)_{x\in \bbR^2}$ be the family of probability measures defined in Corollary \ref{Markovcor}. If $x\in M$ and $m$ is a measure on $M$. We write
    \[ \bbP_m(\cdot) := \int_M \bbP_x(\cdot) \ dm(x). \]
\end{notation}

\begin{definition}[Quasi-stationary measure]
    Let $n\in\bbN, A\in\cB(M)$. A measure $\mu$ on $\cB(M)$ is called a \textit{quasi-stationary measure} for $\varphi$ if
    \[ \bbP_\mu( \varphi_n \in A \ | \ \tn > n ) := \frac{\int_M \bbP_x( \varphi_n\in A) \ d\mu(x)}{\int_M \bbP_x( \varphi_n\in M) \ d\mu(x)} = \mu(A) \]
    where $\tn$ is the stopping time \eqref{stoptime} for the RDS $(\theta,\varphi)$.
\end{definition}

One can think of this measure as an initial distribution of the chain $\varphi$ such that it can evolve for a very long time without being absorbed. Before we can proceed with proving the existence of such a measure for the Brusselator, we expand our framework. Let,
    \[ L^\infty (M) := \{  f:M\to\bbR \ | \ f \text{ is Borel measurable and } ||f||_\infty < \infty \} \]
    where for measurable $f:M\to\bbR$
    \[ ||f||_\infty := \inf \{ C \geq 0 \ | \ |f(x)|\leq C \text{ for almost every } x\in M\}  \]
denotes the \textit{supremum norm}. Furthermore, let
    \[ \cC^0 (M) := \{  f \in L^\infty(M) \ | \ f \text{ is continuous} \}. \]
We state a key theorem in order to use the properties developed in the previous section in the context of measures. For a proof, see \cite[Theorem 6.19]{rudin86real}

\begin{theorem}[Riesz-Markov-Kakutani representation]\label{Rieszthm}
    The topological dual space $(\cC^0(M), \norm{\cdot}_\infty)^* = (\cM(M),\norm{\cdot}_{TV})$ where
    \[ \cM(M) := \{ m: \cB(M) \to \bbR \ | \ m \text{ is a signed measure on } \cB(M) \} \]
    and for all $m\in\cM(M)$
    \[ \norm{m}_{TV} := \abs{m}(M). \]
    In particular, we may identify any $m\in\cM(M)$ with an element of $\cC^0(M)^*$ by
    \[ m(f) = \int_M f(x) \ dm(x )\]
    for all $f\in\cC^0(M).$
\end{theorem}

Using the same notation, we define
\[ \cM_+(M) := \{ m \in \cM(M) \ | \ m(A)\geq 0 \text{ for all } A\in\cB(M) \}. \]

We now introduce a new operator on $\cM(M)$.

\begin{definition}
    Let $n\in\bbN$. The operator $\cL^n:\cM(M)\to\cM(M)$ is the linear operator such that for all $m\in\cM(M)$,
    \[ \cL^n m(A) := \int_M \cP^n(x,A) \ dm(x), \quad \text{and} \quad \cL^n m(f) := m(\cP^n f) = \int_M \cP^n f(x) \ dm(x) \]
    for all $A\in\cB(M)$ and $f\in\cC^0(M).$ Moreover we write $\cL:=\cL^1$.
\end{definition}

Note that $\cL$ is precisely the dual operator of $\cP$ when viewed as an operator on $\cC^0(M)^*$. We exploit the fact $\cL=\cP^*$ in the next results. 

\begin{theorem}\label{quasstatthm}
    Let $\varphi$ be the Brusselator Markov chain. Then,
    \begin{enumerate}[label=(\roman*)]
    \item The stochastic Koopman operator,
    \[ \cP:(\cC^0 (M),\norm{\cdot}_\infty)\to(\cC^0 (M),\norm{\cdot}_\infty), \quad f\mapsto \int_M f(y) \ \cP(\cdot,dy), \]
    is a compact linear operator with spectral radius $r(\cP)=\lambda$.

    \item $\lambda$ is an eigenvalue of $\cP$ with eigenfunction $f_0\in\cC^0(M)\setminus\{0\}$ where $f_0>0$.

    \item $\lambda$ is an eigenvalue of $\cL$ with eigenmeasure $\mu\in\cM(M)\setminus\{0\}$.

    \item $\mu$ is a unique quasi-stationary measure $\mu$ for the chain $\varphi$ with $\supp\mu = M$. 
    \end{enumerate}
\end{theorem}
\begin{proof}
    Under Proposition \ref{propPprop}, $\varphi$ satisfies Hypothesis (H) of \cite{castro2024existence} in the case $Z=\emptyset$. The results follow from \cite[Theorem A and Theorem B]{castro2024existence}.
\end{proof}

Given a quasi-stationary measure, one can derive the existence of a quasi-ergodic measure where expectations of time averages conditioned on survival equal the space average with respect to this measure. We make the following formal definition.

\begin{definition}[Quasi-ergodic measure]\label{quasergdef}
    Let $n\in\bbN$ and $h\in L^\infty(M)$. A measure $\nu$ on $\cB(M)$ is called a \textit{quasi-ergodic measure} for $\varphi$ if
    \[ \lim_{n\to\infty} \bbE_x \left[ \frac{1}{n}\sum_{i=0}^{n-1} h\circ\varphi_i \ \middle| \ \tn>n \right] = \int_M h(y) \ d\nu(y)  \]
    holds for $\nu$-almost every $x\in M$. Again $\tn$ is the stopping time \eqref{stoptime} for $\varphi_n$.
\end{definition}

If a Markov chain admits a quasi-ergodic measure then we say it exhibits quasi-ergodicity. One can think of quasi-ergodicity as an almost local version of ergodicity where paths conditioned on staying within the desired region almost surely visit all of that desired region. Hence if one can describe a process which only talks about paths which stay within the desired region, the quasi-ergodic measure gifts ergodicity to such a process opening a range of powerful theorems. Before we proceed in this fashion for the Brusselator chain $\varphi$, we must first prove it exhibits quasi-ergodicty.

\begin{theorem}\label{quasergthm}
    Let $\varphi$ be the Brusselator Markov chain. Then,
    \begin{enumerate}[label=(\roman*)]
    \item There exists a quasi-ergodic measure $\nu$ for $\varphi$ such that, 
        \[ \nu(A) = \frac{\int_A f_0(y) \ d\mu(y)}{\int_M f_0(y) \ d\mu(y)}  \]
    for every $A\in\cB(M)$ and $\supp \nu = M$.

    \item Given $m$, a Borel measure on M such that $\int_M f_0(y) \ dm(y)>0$ and $\int_M \cP^n(x,M) \ dm(x)>0$ for every $n\in\bbN$, there exists $K(m)>0$ and $\beta>0$, such that
    \[ \norm{\bbP_m(\varphi_n\in\cdot \ | \ \tn>n) -\mu}_{TV} \leq K(m)e^{-\beta n}, \]
    for all $n\in\bbN$.
\end{enumerate}    
\end{theorem}
\begin{proof}
    All results except $\supp \nu = M$ follow from \cite[Theorem C, (M1)]{castro2024existence}. Both $\mu$ and $f_0$ are strictly positive by $\supp\mu=M$ and Theorem \ref{quasstatthm} respectively. Hence, by the explicit form of $\nu$ above, $\nu$ is strictly positive and so $\supp \nu = M$.
\end{proof}

As eluded to before, we seek to study paths corresponding to initial conditions in the set $\Xi\in\cF\otimes\cB(M)$, defined as
\begin{equation}\label{eq:xiset}
    \Xi:= \{ (\omega,x)\in\Sigma\times M \ | \ \tn(\omega,x) = \infty \} = \bigcap_{n\in\bbN} \{ (\omega,x)\in\Sigma\times M \ | \ \tn(\omega,x) > n \}.
\end{equation}

Fortunately, the following theorem allows us to study a process consisting only of such paths. First, we begin by proving a couple properties of the process under quasi-stationarity. For a more comprehensive study of the consequences of quasi-stationarity see the work of Villemonais \cite{M_l_ard_2012,champagnat:hal-00973509}.

\begin{lemma}\label{quasstatlem}
    Let $\varphi$ be the Brusselator Markov chain. Then,
    \begin{enumerate}[label=(\roman*)]
    \item There exists a family of functions $\{\nu_x\}_{x\in M}$ such that
    \[ \sup_{x\in M}\norm{\nu_x}_{TV} \leq 1 + \norm{f_0}_\infty \norm{\mu}_{TV} < \infty \quad \text{and} \quad  \sup_{x\in M}\norm{\cL^n\nu_x}_{TV} = o(\lambda^n), \]
    and for all $x\in M$, the Dirac mass $\delta_x$ at $x$ may be decomposed as,
    \[ \delta_x = f_0(x)\mu + v_x \]

    \item There exists a constant $\gamma>0$ such that,
    \begin{equation}
        \lim_{n\to\infty}\sup_{x\in M}\abs{e^{\gamma n}\cP^n(x,M)-f_0(x)} = 0
    \end{equation}
    \end{enumerate}
\end{lemma}
\begin{proof}
    For (i), see \cite[Proposition 6.5]{castro2024existence}. For (ii), the Markov property of $\varphi$ implies that for all $n,m\in\bbN$,
    \begin{align*}
        \bbP_\mu(\tn>n+m) &= \int_M \bbP_x(\tn>n+m) \ d\mu(x) = \int_M \cP^{n+m}(x,M) \ d\mu(x) \\
        &= \int_M \cP^n(\cP^m(x,M)) \ d\mu(x) = \cL^n \mu(\cP^m\chi_M) = \lambda^n\mu(\cP^m\chi_M) = \lambda^n\mu(M)\cL^m\mu(M) \\
        &= \int_M \cP^n(x,M) \ d\mu(x)\int_M \cP^m(x,M) \ d\mu(x) = \bbP_\mu(\tn>n)\bbP_\mu(\tn>m).
    \end{align*}
    Therefore there exists $\gamma>0$ such that $\bbP_\mu(\tn>n) = e^{-\gamma n}$. Now let $x\in M$. 
    From (iv), $\sup_{x\in M}\norm{\cL^n\nu_x}_{TV} = o(\lambda^n)$ so $\cL^n\nu_x(M) = o(\lambda^n)$. Note that,
    \[ \cP^n(x,M) = \cL^n\delta_x(M) = \cL^n(f_0(x)\mu(M) + \nu_x(M)) = \lambda^nf_0(x)\mu(M) + \cL^n\nu_x(M) = \lambda^nf_0(x) + o(\lambda^n). \] 
    Hence,
    \begin{align*}
        \abs{e^{\gamma n}\cP^n(x,M)-f_0(x)} &= \abs{\frac{\cP^n(x,M)}{\bbP_\mu(\tn>n)}-f_0(x)} = \abs{\frac{\cP^n(x,M)}{\int_M\bbP_x(\tn>n) \ d\mu(x)}-f_0(x)} \\
        &= \abs{\frac{\cP^n(x,M)}{\int_M\cP^n(x,M) \ d\mu(x)}-f_0(x)} = \abs{\frac{\cP^n(x,M)}{\lambda^n}-f_0(x)} \to 0.
    \end{align*}
    as $n\to\infty$. Therefore,
    \[ \lim_{n\to\infty}\sup_{x\in M}\abs{e^{\gamma n}\cP^n(x,M)-f_0(x)} = 0. \]
\end{proof}

We now prove said theorem.

\begin{theorem}\label{Qprocthm}
     Let $\varphi$ be the Brusselator Markov chain. Then,
     \begin{enumerate}[label=(\roman*)]
         \item There exists a family of probability measures $(\bbQ_x)_{x\in M}$ on $(\Sigma\times M, \cF\otimes\cB(M))$ such that for each $x\in M$ and every $m\in\bbN$ and $A_m\in\cF\otimes\cB(M)$,
        \[ \bbQ_x(A) := \frac{e^{\gamma m}}{f_0(x)}\int_A f_0\circ\varphi_m \ d\bbP_x =\lim_{n\to\infty}\bbP_x(A \ | \ \tn>n). \]
        \item The tuple
        \[ \left(\Sigma\times M, (\cF_n\otimes\cB(M))_{n\in\bbNz},(\varphi_n)_{n\in\bbNz},(\cQ^n)_{n\in\bbN},(\bbQ_x)_{x\in M} \right) \]
        is a Markov chain on the state space $(M,\cB(M))$, where the transition kernels $(\cQ^n)_{n\in\bbN}$ satisfy
        \[ \cQ^n(x,A) := \bbQ_x(\varphi_n\in A) = \frac{e^{\gamma n}}{f_0(x)}\int_A f_0(y) \ \cP^n(x,dy) \]
        for all $n\in\bbN, \ x\in M$ and $A\in\cB(M)$. We call this Markov chain a Q-process.
    
        \item The law of the process $(\bbQ_x)_{x\in M}$ gives full measure to orbits of $(\varphi_n)_{n\in\bbN}$ which never get absorbed. In particular, $\bbQ_x(\Xi) = 1$ for all $x\in M$ where $\Xi\in\cF\otimes\cB(M)$ is defined in \eqref{eq:xiset}
        while $\bbP_x(\Xi)=0$ for every $x\in M$.
    
        \item The quasi-ergodic measure $\nu$ is the unique stationary measure for this process with,
        \[ \lim_{n\to\infty}\norm{\cQ^n(x,\cdot)-\nu}_{TV} = 0 \]
        for all $x\in M$.
    \end{enumerate}  
\end{theorem}
\begin{proof}
    Under Theorem \ref{quasergthm} (ii) and Lemma \ref{quasstatlem} (ii), $\varphi$ satisfies Hypothesis (H) of \cite{castro2022lyapunov} and so the claims hold by \cite[Proposition 2.5 and Proposition 3.2]{castro2022lyapunov}.
\end{proof}

The significance of the Q-process is that it describes the dynamics of orbits of $(\theta,\varphi)$ that do not get absorbed by giving full measure to the set of such orbits. Moreover, by the uniqueness of the measure $\bbQ_x$, the Q-process is the only such process which describes the dynamics of the non-absorbing orbits of $\varphi$. As we'd expect, the unique stationary measure for this process coincides with the quasi-ergodic measure $\nu$ for $(\theta,\varphi)$.

Finally we have proved Theorem \ref{quasthm}.

\begin{proof}[Proof of Theorem \ref{quasthm}]
    Most of the results are found in Theorem \ref{quasergthm}, while the last part is contained in Theorem \ref{Qprocthm} (iv)
\end{proof}

\section{The CLE and conditioned Lyapunov exponents}

We move on to our application of the Q-process, namely, conditioned Lypaunov exponents laid out in Definition \ref{lyapexp}. Importantly, we define the skew product and matrix coycle.

\begin{definition}[Skew product]
    The family of mappings $(\Theta_n)_{n\in\bbNz}$ where $\Theta_n:\Sigma\times\bbR^2\to\Sigma\times\bbR^2$ given by
    \[ \Theta_n(\omega,x) = (\theta^n\omega,\varphi_n(\omega,x)) \]
    for all $n\in\bbN$ and $\Theta_0 :=\mathrm{id}_{\Sigma\times\bbR^2}$
    is called the skew product of the RDS $(\theta,\varphi)$. Let $\Theta:=\Theta_1$ and note that under this definition, $\Theta_n$ is precisely the map $\Theta$ composed with itself $n$ times.
\end{definition}

As desired, the following proposition opens our analysis to ergodicity.

\begin{proposition}
    The skew product $(\Theta_n)_{n\in\bbNz}$ of $(\theta,\varphi)$ where $\Theta_n:\Sigma\times\bbR^2\to\Sigma\times\bbR^2$ given by
    \[ \Theta_n(\omega,x) = (\theta^n\omega,\varphi_n(\omega,x)) \]
    for all $n\in\bbN$ and $\Theta_0 :=\mathrm{id}_{\Sigma\times\bbR^2}$ forms a metric DS on $\Sigma\times M$ with respect to the probability measure $\bbQ_\nu:\cF\otimes\cB(M)\to[0,1]$ with
    \[ \bbQ_\nu(A) := \int_M \bbQ_x(A) \ d\nu(x) \]
    for all $A\in\cF\otimes\cB(M)$.
    
    In addition, $\bbQ_\nu$ is ergodic with respect to $(\Theta_n)_{n\in\bbNz}$ and gives full measure to orbits of $(\varphi_n)_{n\in\bbN}$ which never get absorbed. In particular, $\bbQ_\nu(\Xi) = 1$
    
     We denote the expectation with respect to the measure $\bbQ_\nu$ by $\bbE^\bbQ_\nu$.
\end{proposition}
\begin{proof}
    See \cite[Proposition 2.6]{castro2022lyapunov}.
\end{proof}

For the following definition note that the Brusselator RDS $(\theta,\varphi)$ according to \eqref{varphi}, \eqref{eq:fbruss} and \eqref{eq:Fbrus} is of class $\cC^1$.

\begin{definition}[Matrix cocycle]\label{matcyc}
    Let $x\in M$ and $T_xM\subset\bbR^2$ denote the tangent space of $M$ at $x$. For each $(\omega,x)\in\Sigma\times M$ and $n\in\bbN$ define the linear map $\Phi_n(\omega,x):T_xM \to T_{\varphi_n(\omega,x)}M$ given by
        \[ \Phi_n(\omega,x)v = A(\Theta_{n-1}(\omega,x))...A(\Theta(\omega,x))A(\omega,x)v \]
    for all $v\in T_xM$. Here $(\Theta_n)_{n\in\bbNz}$ is the skew product of $(\theta,\varphi)$ and $A:\Sigma\times M\to\bbR^{2\times2}$ is given by,
    \[ A(\omega,x) := \mathrm{D}_x\varphi_1(\omega,x) = \begin{pmatrix}
        A_{11}(\omega,x) & A_{12}(\omega,x) \\
        A_{21}(\omega,x) & A_{22}(\omega,x) \end{pmatrix} \]
    where $\mathrm{D}_x$ denotes the derivative operator with respect to $x$ and
        \begin{align*}
            A_{11}(\omega,x) &:= 1+(2x_1x_2-(1+b))\tau + \sqrt{\frac{\tau}{\Omega}}\left(-\frac{1}{2\sqrt{x_1}}\omega_{0,2}-\frac{1}{2}\sqrt{\frac{b}{x_1}}\omega_{0,3}+\sqrt{x_2}\omega_{0,4}\right), \\
            A_{12}(\omega,x) &:= x_1^2\tau + \sqrt{\frac{\tau}{\Omega}}\left(\frac{x_1}{2\sqrt{x_2}}\omega_{0,4}\right), \\
            A_{21}(\omega,x) &:= (b-2x_1x_2)\tau + \sqrt{\frac{\tau}{\Omega}}\left(\frac{1}{2}\sqrt{\frac{b}{x_1}}\omega_{0,3}-\sqrt{x_2}\omega_{0,4}\right), \\
            A_{22}(\omega,x) &:= 1 - x_1^2\tau - \sqrt{\frac{\tau}{\Omega}}\left(\frac{x_1}{2\sqrt{x_2}}\omega_{0,4}\right).
        \end{align*}
        We have used the notation $x = (x_1,x_2)^\intercal$ and $\omega = (\omega_0,\omega_1,...) = ((\omega_{0,1},...,\omega_{0,4})^\intercal, (\omega_{1,1},...,\omega_{1,4})^\intercal,...)$ and refer to $\Phi$ as the matrix cocycle.
\end{definition}

The matrix cocycle $\Phi$ defines the evolution of the tangent vectors to those orbits which start at $x\in M$ assuming they haven't absorbed by the cemetery state. In particular, for $\omega\in\Sigma, \ x\in M$ and $\tn>n$, the matrix cocycle at time $n$ is precisely $\Phi_n(\omega,x) = \mathrm{D}_x\varphi_n(\omega,x)$

Castro, Cheminitz, Chu, Engel, Lamb and Rasmussen \cite{castro2022lyapunov} proved results analogous to Theorem \ref{osethm} and more for the abstract case of RDS exhibiting an invertible matrix cocycle $\Phi$ and a certain integrability condition. Unfortunately, for our Brusselator RDS $(\theta,\varphi)$, we do not have the luxury of such invertibility and hence we cannot cite the results we need directly from \cite{castro2022lyapunov}. Instead we follow their techniques closely and exploit the explicit form of $\Phi$ in order to work our way around the lack of invertibility.

First we require the following lemma.

\begin{lemma}\label{explem}
    Let $\Phi$ be the matrix cocycle of $(\theta,\varphi)$. Then for $\bbQ_x$-almost all $x\in M$ and $n\in\bbN$, there exists a constant $\Tilde{C}>0$ such that,
    \[ \bbE_x\left[\abs{\log\norm{\Phi_n}}^2\right] \leq \Tilde{C}n^2 \quad \text{and} \quad \bbE_x\left[\abs{\log\abs{\det\Phi_n}}^2\right] \leq 4\Tilde{C}n^2. \]
\end{lemma}
\begin{proof}
    Let $k\in\bbNz$ and $x\in M$. We first seek to bound the term, $\bbE_x\left[\norm{A\circ\Theta_k}\right]$
    Since all norms on a finite dimensional space are equivalent, there exists a constant $L>0$ such that
    \[ \norm{A(\Theta_k(\omega,x))} = \norm{A(\theta^k\omega,\varphi_k(\omega,x))} \leq L \sum_{i,j}\abs{A_{ij}(\theta^k\omega,\varphi_k(\omega,x))}. \]
    for each $k\in\bbNz$ and $(\omega,x)\in\Sigma\times M$. Next we may deduce, from the explicit form of the $A_{ij}$ in Definition \ref{matcyc} and repeated use of the triangle inequality, that each $\abs{A_{ij}(\theta^k\omega,\varphi_k(\omega,x))}$ is bounded above by non-negative linear combinations of $\abs{\varphi_k(\omega,x)_1}^p\abs{\varphi_k(\omega,x)_2}^q\abs{\omega_{k,m}}^r$ and  where $\varphi_k(\omega,x) = (\varphi_k(\omega,x)_1,\varphi_k(\omega,x)_2)^\intercal$ with $p,q\in\{-\frac{1}{2},0,\frac{1}{2},1,2\}, \ r\in\{0,1\}$ and $m\in\{1,...,4\}$. That is,
    \[ \abs{A_{ij}(\theta^k\omega,\varphi_k(\omega,x))} \leq \sum_{p,q}\sum_{m=1}^4\sum_{r=0}^1 a_{ij}(m,r,p,q) \abs{\varphi_k(\omega,x)_1}^p\abs{\varphi_k(\omega,x)_2}^q\abs{\omega_{k,m}}^r  \]
    for each $i,j\in\{1,2\}$ where each constant $a_{ij}(m,r,p,q)\geq0$ depends on $m,r,p,q$ but is independent of $k$.
    Also, since $\varphi_k(\omega,x)\in M \ \bbQ_x$-a.s for all $k\in\bbNz, \ (\omega,x)\in \Sigma\times M$ and $M\subset\bbR^2_{>0}$ is compact, then there exists constants $C,c>0$ such that $0<c<\varphi_k(\omega,x)_1,\varphi_k(\omega,x)_2<C$ for all $k\in\bbNz$. It follows that, regardless of the sign of the powers $p$ and $q$ and then, $\abs{\varphi_k(\omega,x)_1}^p\abs{\varphi_k(\omega,x)_2}^q$ is bounded above and has the same upper bound for each $k\in\bbNz$. Hence we may write,
    \begin{align*}
        \abs{A_{ij}(\theta^k\omega,\varphi_k(\omega,x))} &\leq \sum_{p,q}\sum_{m=1}^4\sum_{r=0}^1 a_{ij}(m,r,p,q) \abs{\varphi_k(\omega,x)_1}^p\abs{\varphi_k(\omega,x)_2}^q\abs{\omega_{k,m}}^r \\
        &\leq \sum_{p,q}\sum_{m=1}^4\sum_{r=0}^1 a_{ij}(m,r,p,q) b_{ij}(m,r,p,q) \abs{\omega_{k,m}}^r,
    \end{align*} 
    for constants $ b_{ij}(m,r,p,q)\geq0$ which are also dependent on $m,r,p,q$ but independent of $k$. 

    Since $\omega_{k,m}$ is normally distributed with zero mean and unit variance for all $k\in\bbNz$ and $m\in\{1,...,4\}$, we let $\Tilde{\eta}$ be the measure corresponding to this distribution. 
    \begin{gather*}
        \int_\Sigma \abs{A_{ij}(\theta^k\omega,\varphi_k(\omega,x))} \ d\bbP(\omega) \leq \int_\bbR \sum_{p,q}\sum_{m=1}^4\sum_{r=0}^1 a_{ij}(m,r,p,q) b_{ij}(m,r,p,q) \abs{y}^r \ d\Tilde{\eta}(y) \\
    = \sum_{p,q}\sum_{m=1}^4 a_{ij}(m,0,p,q) b_{ij}(m,0,p,q)
    + \sum_{p,q}\sum_{m=1}^4 a_{ij}(m,1,p,q) b_{ij}(m,1,p,q) \int_\bbR \abs{y} \ d\Tilde{\eta}(y)
    \end{gather*}
    
    Then by the symmetry of the standard normal distribution,
    \[ \int_\bbR \abs{y} \ d\Tilde{\eta}(y) = 2 \int_{[0,\infty)} y \ d\Tilde{\eta}(y) = \frac{2}{\sqrt{2\pi}} \int_{[0,\infty)} y e^{-\frac{y^2}{2}} \ d\ell(y) = \sqrt{\frac{2}{\pi}}. \]
    Consequently,
    \begin{multline*}
        \int_\Sigma \abs{A_{ij}(\theta^k\omega,\varphi_k(\omega,x))} \ d\bbP(\omega) \leq \sum_{p,q}\sum_{m=1}^4 a_{ij}(m,0,p,q) b_{ij}(m,0,p,q) \\
        + \sum_{p,q}\sum_{m=1}^4 a_{ij}(m,1,p,q) b_{ij}(m,1,p,q)\sqrt{\frac{2}{\pi}}. 
    \end{multline*}
    This sum is finite as it contains finitely many terms which themselves are finite. Crucially, the sum is also independent of $k$. Let
    \[ C_{ij}:= \sum_{p,q}\sum_{m=1}^4 a_{ij}(m,0,p,q) b_{ij}(m,0,p,q) + \sum_{p,q}\sum_{m=1}^4 a_{ij}(m,1,p,q) b_{ij}(m,1,p,q)\sqrt{\frac{2}{\pi}}<\infty. \]
    It follows that,
    \begin{align*}
        \bbE_x\left[\norm{A\circ\Theta_k}\right] &= \int_{\Sigma\times M} \norm{A(\Theta_k(\omega,y))} \ d\bbP_x(\omega,y) = \int_\Sigma \norm{A(\Theta_k(\omega,x))} \ d\bbP(\omega) \\
        &\leq \int_\Sigma L \sum_{i,j}\abs{A_{ij}(\theta^k\omega,\varphi_k(\omega,x))} \ d\bbP(\omega) = L \sum_{i,j} \int_\Sigma \abs{A_{ij}(\theta^k\omega,\varphi_k(\omega,x))} \ d\bbP(\omega) \\
        &\leq L \sum_{i,j} C_{ij} < \infty.
    \end{align*}

    We now bound $\bbE_x\left[\abs{\log\norm{\Phi_n}}^2\right]$ for each $n\in\bbN$. By Definition \ref{matcyc}, for all $k\in\bbNz, \ \omega\in\Sigma$ and $x\in M$,  $A_{12}(\Theta_k(\omega,x))+A_{22}(\Theta_k(\omega,x)) = 1$ and so $\max\{\abs{A_{12}(\Theta_k(\omega,x))},\abs{A_{22}(\Theta_k(\omega,x))}\} \geq 1/2$ . By the equivalence of norms on finite dimensional spaces, there is a constant $R>0$, such that for all $k\in\bbNz, \ \omega\in\Sigma$ and $x\in M$,
    \[ \norm{A(\Theta_k(\omega,x))} \geq R\max_{i,j}\abs{A_{ij}(\Theta_k(\omega,x))} \geq \frac{R}{2}. \]
   
    A simple calculus argument can be used to show that $f:[1,\infty)\to\bbR$ given by $f(x) = x - (\log x)^2$ is an increasing function with $f(1) = 1$. And so $(\log x)^2<x$ for all $x\geq 1$. In particular, if $R\geq 2$, then for all $k\in\bbNz, \ \omega\in\Sigma$ and $\bbQ_x$-almost all $x\in M$,
    \[ \bbE_x\left[\abs{\log\norm{A\circ\Theta_k}}^2 \right] = \bbE_x\left[ \chi_{\{\norm{A\circ\Theta_k} \geq 1\}} \abs{\log\norm{A\circ\Theta_k}}^2 \right] \leq \bbE_x\left[ \norm{A\circ\Theta_k} \right] \leq L \sum_{i,j} C_{ij}. \]
    
    If $R<2$, then
    \[ \bbE_x\left[ \chi_{\{R/2 \leq \norm{A\circ\Theta_k} < 1\}} \abs{\log\norm{A\circ\Theta_k}}^2 \right] \leq \bbE_x\left[ \abs{\log\frac{R}{2}}^2 \right] = \abs{\log\frac{R}{2}}^2 \]
    since $\abs{\log}^2$ is decreasing on $[\frac{R}{2},1)$. Hence,
    \begin{align*}
        \bbE_x\left[ \abs{\log\norm{A\circ\Theta_k}}^2 \right] &= \bbE_x\left[ \chi_{\{R/2 \leq \norm{A\circ\Theta_k} < 1\}} \abs{\log\norm{A\circ\Theta_k}}^2 \right] + \bbE_x\left[ \chi_{\{\norm{A\circ\Theta_k} \geq 1\}} \abs{\log\norm{A\circ\Theta_k}}^2 \right] \\
        &\leq \abs{\log\frac{R}{2}}^2 + L \sum_{i,j} C_{ij}.
    \end{align*}
    
    Let $\Tilde{C}:= \abs{\log\frac{R}{2}}^2 + L \sum_{i,j} C_{ij}$, so that, no matter the size of $R$, $\bbE_x\left[ \abs{\log\norm{A\circ\Theta_k}}^2 \right] \leq \Tilde{C}$ for all $k\in\bbNz$.

    Now let $n\in\bbN$. By Definition \ref{matcyc},
    \[ \norm{\Phi_n(\omega,x)} = \norm{A(\Theta_{n-1}(\omega,x))...A(\Theta(\omega,x))A(\omega,x)} \leq \norm{A(\Theta_{n-1}(\omega,x))}...\norm{A(\Theta(\omega,x))}\norm{A(\omega,x)} \]
    for all $(\omega,x)\in\Sigma\times M$ and $n\in\bbN$. It follows from the additive property of $\log$ and the Cauchy-Schwarz inequality \cite[Theorem 1.35]{rudin1976principles} that,
    \begin{align*}
         \bbE_x\left[\abs{\log\norm{\Phi_n}}^2\right] &\leq \bbE_x\left[\left( \sum_{k=0}^{n-1} \abs{\log\norm{A\circ\Theta_k}}\right)^2\right] 
         \leq \bbE_x\left[ n\cdot \sum_{k=0}^{n-1} \abs{\log\norm{A\circ\Theta_k}}^2\right] \\
         &= n\cdot \sum_{k=0}^{n-1} \bbE_x\left[\abs{\log\norm{A\circ\Theta_k}}^2\right] \leq  n\cdot \sum_{k=0}^{n-1} \Tilde{C} = \Tilde{C}n^2.
    \end{align*}
    for all $n\in\bbN$ and $\bbQ_x$-almost all $x\in M$.
    
    In addition, note that for all $(\omega,x)\in\Sigma\times M$ and $n\in\bbN$,
    \[ \abs{\det\Phi_n(\omega,x)} \leq \norm{\Phi_n(\omega,x)}^2 \]
    when $\norm{\cdot}$ is induced by the Euclidean norm and so,
    \[ \bbE_x\left[\abs{\log\abs{\det\Phi_n}}^2\right] \leq \bbE_x\left[\abs{\log\left(\norm{\Phi_n}^2\right)}^2\right] = \bbE_x\left[\abs{2\log\norm{\Phi_n}}^2\right] = 4\bbE_x\left[\abs{\log\norm{\Phi_n}}^2\right] \leq 4\Tilde{C}n^2. \]
\end{proof}

We are now able to apply Kingman's subadditive ergodic theorem in order to prove a version of the Furstenberg-Kesten theorem regarding the existence of the limits in Theorem \ref{osethm}. 

\begin{theorem}[Furstenberg-Kesten]\label{furstkest}
    Let $\Phi$ be the matrix cocycle of $(\theta,\varphi)$. There exists a $\Theta$-forward invariant set $\Delta\in\cF\otimes\cB(M), \ \Delta\subseteq\Xi$ such that $\bbQ_\nu(\Delta)=1$ and the limits 
    \[ \Lambda_1 := \lim_{n\to\infty}\frac{1}{n}\log \norm{\Phi_n(\omega,x)} \quad \text{and} \quad \Lambda_2:=\lim_{n\to\infty}\frac{1}{n}\log \abs{\det\Phi_n(\omega,x)} \]
    exist $\bbQ_\nu$-a.s for all $(\omega,x)\in\Delta$. Moreover $\Lambda_1, \Lambda_2 > -\infty$ are constant as well as both
    \[ \lim_{n\to\infty} \bbE^\bbQ_\nu \left[ \abs{\Lambda_1 - \frac{1}{n}\log \norm{\Phi_n} } \right] = 0 \quad \text{and} \quad \lim_{n\to\infty} \bbE^\bbQ_\nu \left[ \abs{\Lambda_2 - \frac{1}{n}\log \abs{\det\Phi_n} } \right] = 0. \] 
\end{theorem}
\begin{proof}
    Let $n\in\bbN$ and $\phi_n,\psi_n:\Sigma\times M\to\bbR, \ \phi_n(\omega,x)=\log\norm{\Phi_n(\omega,x)}$ and $\psi_n(\omega,x)=\log\abs{\det\Phi_n(\omega,x)}$ for all $(\omega,x)\in\Sigma\times M$. We show $(\phi_n)_{n\in\bbN}$ and $(\psi_n)_{n\in\bbN}$ are subadditive over the metric DS $(\Sigma\times M,\cF\otimes\cB(M),\bbQ_\nu, (\Theta_n)_{n\in \bbNz})$. Let $\omega\in\Sigma$ and $x\in M$. Then, 
    \[ \Phi_{n+m}(\omega,x) = A(\omega,x)...A(\Theta_{m-1}(\omega,x))A(\Theta_m(\omega,x))...A(\Theta_{n-1}(\Theta_m(\omega,x))) = \Phi_m(\omega,x)\Phi_n(\Theta_m(\omega,x)) \]
    and so,
    \begin{align*}
        \phi_{n+m}(\omega,x) &= \log\norm{\Phi_{n+m}(\omega,x)} = \log\norm{\Phi_m(\omega,x)\Phi_n(\Theta_m(\omega,x))} \leq  \log(\norm{\Phi_m(\omega,x)}\norm{\Phi_n(\Theta_m(\omega,x))}) \\
        &= \log\norm{\Phi_m(\omega,x)}+\log\norm{\Phi_n(\Theta_m(\omega,x))} = \phi_m(\omega,x) + \phi_n(\Theta_m(\omega,x)),
    \end{align*}
    and
    \begin{align*}
        \psi_{n+m}(\omega,x) &= \log\abs{\det\Phi_{n+m}(\omega,x)} = \log\abs{\det(\Phi_m(\omega,x)\Phi_n(\Theta_m(\omega,x)))} \\
        &=  \log(\abs{\det\Phi_m(\omega,x)}\abs{\det\Phi_n(\Theta_m(\omega,x))}) 
        = \log\abs{\det\Phi_m(\omega,x)}+\log\abs{\det\Phi_n(\Theta_m(\omega,x))} \\
        &= \psi_m(\omega,x) + \psi_n(\Theta_m(\omega,x)),
    \end{align*}
    as required.

    From Lemma \ref{explem}, 
    \[ \bbE_x\left[\abs{\phi_n}^2\right] = \bbE_x\left[\abs{\log\norm{\Phi_n}}^2\right] \leq \Tilde{C}n^2 \]
    and,
    \[ \bbE_x\left[\abs{\psi_n}^2\right] = \bbE_x\left[\abs{\log\abs{\det\Phi_n}}^2\right] \leq 4\Tilde{C}n^2 \]
    for $\bbQ_x$-almost all $x\in M$. By Jensen's inequality \cite[Proposition 1.7.3]{ross1996stochastic}
    \[ \left(\bbE_x\left[\abs{\phi_n}\right]\right)^2 \leq \bbE_x\left[\abs{\phi_n}^2\right] \quad \text{and} \quad \left(\bbE_x\left[\abs{\psi_n}\right]\right)^2 \leq \bbE_x\left[\abs{\psi_n}^2\right].  \]
    Hence by combining these two facts,
    \[ \bbE_x\left[\abs{\phi_n}\right] \leq n\sqrt{\Tilde{C}} \quad \text{and} \quad \bbE_x\left[\abs{\psi_n}\right] \leq 2n\sqrt{\Tilde{C}}. \]
    It follows that since $\phi_n$ is $(\cF_n\otimes\cB(M))$-measurable,
    \begin{align*}
        \int_{\Sigma\times M} \abs{\phi_n} \ d\bbQ_x &= \frac{e^{\gamma n}}{f_0(x)}\int_{\Sigma\times M} \abs{\log\norm{\Phi_n(\omega,y)}}f_0(\varphi_n(\omega,y)) \ d\bbP_x(\omega,y) \\
        &\leq \frac{e^{\gamma n}\norm{f_0}_\infty}{f_0(x)}\bbE_x\left[\abs{\phi_n}\right] \leq \frac{e^{\gamma n}\norm{f_0}_\infty}{f_0(x)}n\sqrt{\Tilde{C}}.
    \end{align*}
    Since $f_0\in\cC^0_+(M), \ f_0>0$ and $M$ is compact, then there exists a constant $c_1>0$ such that $f_0(x)\geq c_1$ for all $x\in M$.
    Hence,
    \[ \int_{\Sigma\times M} \abs{\phi_n} \ d\bbQ_\nu = \int_M\int_{\Sigma\times M} \abs{\phi_n} \ d\bbQ_x d\nu(x) \leq \int_M \frac{e^{\gamma n}\norm{f_0}_\infty}{f_0(x)}n\sqrt{\Tilde{C}} \ d\nu(x) \leq \frac{e^{\gamma n}\norm{f_0}_\infty}{c_1}n\sqrt{\Tilde{C}} < \infty. \]
    Moreover, by definition $\cP^n(x,M)>0$ for all $x\in M$ and $n\in\bbN$ so there exists a constant $c_2>0$ such that  $\cP^n(x,M)>c_2$ for all $n\in\bbN$. Now by Proposition \ref{quasstatlem} (ii), $\limsup_{n\to\infty} \frac{e^{\gamma n}}{f_0(x)} = \limsup_{n\to\infty}\frac{1}{\cP^n(x,M)}<\frac{1}{c_2}$ for all $x\in M$. Then using this fact and Fatou's lemma \cite[Lemma 1.28]{rudin86real},
    \begin{align*}
        \limsup_{n\to\infty}\frac{1}{n}\int_{\Sigma\times M} \abs{\phi_n} \ d\bbQ_\nu  &= \limsup_{n\to\infty}\frac{1}{n}\int_M \int_{\Sigma\times M} \abs{\phi_n} \ d\bbQ_x d\nu(x) \leq \int_M \limsup_{n\to\infty}\frac{1}{n} \int_{\Sigma\times M} \abs{\phi_n} \ d\bbQ_x d\nu(x) \\ 
        &\leq \int_M \limsup_{n\to\infty} \frac{e^{\gamma n}\norm{f_0}_\infty}{f_0(x)}\sqrt{\Tilde{C}} \ d\nu(x) \leq \int_M \frac{\norm{f_0}}{c_2}\sqrt{\Tilde{C}} \ d\nu(x) = \frac{\norm{f_0}}{c_2}\sqrt{\Tilde{C}} < \infty.
    \end{align*}
    Consequently,
    \[ \sup_{n\in\bbN}\frac{1}{n}\int_{\Sigma\times M} \abs{\phi_n} \ d\bbQ_\nu < \infty \quad \text{and} \quad \inf_{n\in\bbN}\frac{1}{n}\int_{\Sigma\times M} \phi_n \ d\bbQ_\nu \geq - \sup_{n\in\bbN}\frac{1}{n}\int_{\Sigma\times M} \abs{\phi_n} \ d\bbQ_\nu > -\infty  \]
    
    Through the same argument but with $\psi_n$ in place of $\phi_n$ and $2\sqrt{\Tilde{C}}$ in place of $\sqrt{\Tilde{C}}$,
    \[ \int_{\Sigma\times M} \abs{\psi_n} \ d\bbQ_\nu < \infty \quad \text{and} \quad \inf_{n\in\bbN}\frac{1}{n}\int_{\Sigma\times M} \psi_n \ d\bbQ_\nu >-\infty. \]

    The result now follows immediately from Kingman's Subadditive Ergodic Theorem \cite[Section 1.5]{Krengel+1985} with $\phi_n(\omega,x)=\log\norm{\Phi_n(\omega,x)}, \ \phi(\omega,x) = \Lambda_1$ and $\psi_n(\omega,x)=\log\abs{\det\Phi_n(\omega,x)}, \ \psi(\omega,x) = \Lambda_2$.
\end{proof}

Under the same setting, we prove the convergence of the previous limits in conditional probability. 

\begin{proposition}\label{condexpprop}
    Let $\Phi$ be the matrix cocycle of $(\theta,\varphi)$ and let $\Lambda_1$ and $\Lambda_2$ be the constants defined in the Furstenberg-Kesten Theorem \ref{furstkest}. Then for $\nu$-almost all $x\in M$,
    \[ \lim_{n\to\infty} \bbE_x \left[ \abs{\Lambda_1- \frac{1}{n}\log\norm{\Phi_n}} \ \middle| \ \tn>n \right] = 0 \quad \text{and} \quad \lim_{n\to\infty} \bbE_x \left[ \abs{\Lambda_2 - \frac{1}{n}\log\abs{\det\Phi_n}} \ \middle| \ \tn>n \right] = 0.\] \
\end{proposition}
\begin{proof}
    In the same way as the proof of Theorem $\ref{furstkest}$, let $n\in\bbN$ and $\phi_n,\psi_n:\Sigma\times M\to\bbR, \ \phi_n(\omega,x)=\log\norm{\Phi_n(\omega,x)}$ and $\psi_n(\omega,x)=\log\abs{\det\Phi_n(\omega,x)}$ for all $(\omega,x)\in\Sigma\times M$.

    Let $x\in M$. Note that in Theorem \ref{furstkest} we showed that $\bbE^\bbQ_x\left[\abs{\phi_n}\right]<\infty$. Hence,
    \[ \abs{ \bbE^\bbQ_x\left[\Lambda_1-\frac{\phi_n}{n}\right]} \leq \Lambda_1 + \frac{1}{n}\bbE^\bbQ_x\left[\abs{\phi_n}\right] < \infty.   \]
    It follows from the dominated convergence theorem \cite[Theorem 1.34]{rudin86real} and Theorem \ref{furstkest} that, 
    \[  \int_M \lim_{n\to\infty} \bbE^\bbQ_x\left[\Lambda_1-\frac{\phi_n}{n}\right] \ d\nu(x) = \lim_{n\to\infty} \int_M \bbE^\bbQ_x\left[\Lambda_1-\frac{\phi_n}{n}\right] \ d\nu(x) = \lim_{n\to\infty} \bbE^\bbQ_\nu\left[\Lambda_1-\frac{\phi_n}{n}\right] = 0. \]
    Consequently for $\nu$-almost every $x\in M$,
    \[ \lim_{n\to\infty} \bbE^\bbQ_x\left[\Lambda_1-\frac{\phi_n}{n}\right] = 0. \]

    Moreover, by definition $\cP^n(x,M)>0$ for all $x\in M$ and $n\in\bbN$ so there exists a constant $c'>0$ such that  $\cP^n(x,M)>c'$ for all $n\in\bbN$. Now by Lemma \ref{explem}, 
    \[ \bbE_x\left[ \abs{\frac{\phi_n}{n}}^2 \ \middle| \ \tn>n \right] = \frac{1}{n^2}\frac{\bbE_x\left[\abs{\phi_n}^2\right]}{\cP^n(x,M)} \leq \frac{1}{n^2}\frac{\Tilde{C}n^2}{c'} = \frac{\Tilde{C}}{c'} < \infty. \]
    
    Now the first limit in the proposition follows from \cite[Theorem 2.10]{castro2022lyapunov} where $p=2$. The second limit follows from using the same argument but with $\psi_n$ in place of $\phi_n$ and $4\Tilde{C}$ in place of $\Tilde{C}$.
\end{proof}

Having established the finite nature of the limits and the convergence in conditional expectation we are able to prove Theorem C.

\begin{proof}[Proof of Theorem \ref{osethm}]
    Since $\Phi$ forms a linear cocyle over the ergodic dynamical system $(\Sigma\times M, \cF\otimes\cB(M),\bbQ_\nu,(\Theta_n)_{n\in\bbNz})$ and $\bbE^\bbQ_\nu\left[ \log\norm{A}\right] = \bbE^\bbQ_\nu\left[ \log\norm{\Phi_0}\right] < \infty$
    (as shown in the proof of Theorem \ref{furstkest}),
    we can apply \cite[Theorem 3.4.1 (A)]{arnold2013random}. The result follows from noting that $\lambda_1=\Lambda_1$ and $\lambda_2=\Lambda_2-\Lambda_1$.
\end{proof}

\section{The CLE and synchronisation}

For the following results, we let  $\bbP^\bbQ_\nu$ be the measure on $\Sigma$ such that for all $x\in M$, $\bbQ_x = \bbP^\bbQ_\nu\otimes \delta_x$. In order to prove Theorem \ref{syncthm}, we need the following lemma. 

\begin{lemma}\label{synclem}
    Let $(\theta,\varphi)$ be the Brusselator RDS. Then, $(\theta,\varphi)$ is two point contractible on $M$ in the sense of \cite[Definition 3.1.1]{NEWMAN_2018}. That is, for all $x,y\in M$ and $\varepsilon>0$,
    \[ \bbP^\bbQ_\nu(\{\omega \in \Sigma \ | \ \exists n\in\bbNz \text{ such that } \norm{\varphi_n(\omega,x)-\varphi_n(\omega,y)}<\varepsilon\}) > 0. \]
\end{lemma}
\begin{proof}
    Let $x,y\in M$ and $\varepsilon>0$. Consider the map $g:\bbR^4\to\bbR$ such that,
    \[ g(t) = \norm{F_0(x)-F_0(y) + (F(x)-F(y))t} \]
     and define $U\in\cF$ to be the following cylinder set,
    \[ U := \left[ g^{-1}B(0,\varepsilon) \right]. \]
    In other words, $U$ is the cylinder set of sequences $(\omega_n)_{n\in\bbNz}\in\Sigma$ whose first entry $\omega_0$ is contained in the preimage under $g$ of the open ball of radius $\varepsilon$, centered at the origin. By continuity of $g$, the set $g^{-1}B(0,\varepsilon)\subset \bbR^4$ is open. Hence there exists a $z\in g^{-1}B(0,\varepsilon)$ and $\delta>0$ such that, $B(y,\delta)\subset A$ where $B(z,\delta)$ denotes the open ball of radius $\delta$ centered at $z$. Choose $m\in\bbN$. Since $f_0>0$, then there exists a $K>0$ such that $\frac{e^{\gamma m}}{f_0}f_0\circ\varphi_m\geq K$ and so,
    \begin{align*}
        \bbP^\bbQ_\nu(U) &= \int_M \bbQ_{x'}(U\times M) \ d\nu(x') = \int_M \frac{e^{\gamma m}}{f_0(x')}\int_{U\times M} f_0\circ\varphi_m \ d\bbP_{x'} d\nu(x') \\
        &= \int_M \frac{e^{\gamma m}}{f_0(x')}\int_{U} (f_0\circ\varphi_m)(\omega,x') \ d\bbP(\omega) d\nu(x') 
        \geq  \int_M\int_{g^{-1}B(0,\varepsilon)} K \ d\eta(\omega_0)d\nu(x') \\
        &\geq \frac{K}{4\pi^2}\int_{B(z,\delta)} e^{-\frac{1}{2}\omega_0^\intercal \omega_0} \ d\ell^4(\omega_0) > 0
    \end{align*}
    
    The result follows by noting that,
    \begin{gather*}
        \bbP^\bbQ_\nu(\{\omega \in \Sigma \ | \ \exists n\in\bbNz \text{ such that } \norm{\varphi_n(\omega,x)-\varphi_n(\omega,y)}<\varepsilon\}) \geq \bbP^\bbQ_\nu(\{\omega \in \Sigma \ | \ \norm{\varphi_1(\omega,x)-\varphi_1(\omega,y)}<\varepsilon\}) \\
        = \bbP^\bbQ_\nu(\{\omega \in \Sigma \ | \ \norm{F_0(x)-F_0(y) + (F(x)-F(y))\omega_0}<\varepsilon\}) 
        =  \bbP^\bbQ_\nu(\{\omega \in \Sigma \ | \ \omega_0\in g^{-1}B(0,\varepsilon)\}) = \bbP^\bbQ_\nu(U) > 0.
    \end{gather*}
\end{proof}

Using Lemma \ref{synclem}, we are able to pass the synchronisation of neighbouring paths, implied by negativity of the top exponent, to synchronisation of almost all paths across $M$. 

\begin{proof}[Proof of Theorem \ref{syncthm}]
    From \cite[Remark 3.2.2]{NEWMAN_2018}, Lemma \ref{synclem} and Theorem \ref{Qprocthm} (iv) imply that $M=\supp \nu$ is the unique minimal set. By \cite[Remark 2.2.12]{NEWMAN_2018}, $\lambda_1<0$ implies $\bbQ_\nu$-almost every point in $M$ is almost surely stable in the sense of \cite[Definition 2.2.6]{NEWMAN_2018} and hence admits stable trajectories by \cite[Lemma 2.2.13]{NEWMAN_2018}. Hence the results follow from \cite[Theorem 3.2.1 and Remark 2.3.2]{NEWMAN_2018}
\end{proof}

\section{The LNA and stochastic flows}

Having proven all we stated for the CLE, we move on to the LNA. In order to characterise the process $\{X^\LNA(t) \ | \ t\geq 0\}$, we seek the work of Kunita \cite{Kunita_1990}; a systematic treatment of the theory of stochastic differential equations, stochastic flow of diffeomorphisms and their properties.

\begin{proof}[Proof of Theorem \ref{stochdiffthm}]
    Let $\{z(t) \ | \ t\geq 0, \ z(0) = z_0\}$ be a solution of the RRE \eqref{eq:RRE} for some $z_0\in\bbR^2$. Since $X^\LNA(t) = z(t) + \Omega^{-1/2}\xi(t)$ for each $t\geq 0$, then $X^\LNA$ solves,
    \begin{align*}
        dX^\LNA(t) &= dz(t) + \Omega^{-1/2}d\xi(t) \\
        &= \left(\sum_{j=1}^r \alpt_j(z(t))\nu_j + \Omega^{-1/2}\sum_{j=1}^r \nu_j \mathrm{D}_z \alpt_j(z(t))\xi(t) \right) dt + \Omega^{-1/2}\sum_{j=1}^r \nu_j \sqrt{\alpt_j(z(t))} \ dW^j_t.
    \end{align*}
    The second equality follows from substitution of \eqref{eq:RRE} and \eqref{eq:xiSDE}. Noting that $\Omega^{-1/2}\xi(t) = X^\LNA(t)-z(t)$,
    \begin{multline}\label{eq:stochflowSDE}
         dX^\LNA(t) =  \left( \sum_{j=1}^r\nu_j \mathrm{D}_z \alpt_j(z(t))X^\LNA(t) + \sum_{j=1}^r\left(\alpt_j(z(t))\nu_j - \nu_j \mathrm{D}_z \alpt_j(z(t))z(t)\right)\right)dt \\
        +  \Omega^{-1/2}\sum_{j=1}^r \nu_j \sqrt{\alpt_j(z(t))} \ dW^j_t.
    \end{multline}
    Hence the process $\{X^\LNA(t) \ | \ t\geq 0\}$ solves the continuous non-autonomous linear SDE \eqref{eq:stochflowSDE} and so, by Kunita \cite{Kunita_1990}, generates $\{\Psi_{s,t}^{z_0} (\omega^*) \ | \ \omega^*\in\Sigma^*, \ s,t\geq 0\}$, a two-parameter stochastic flow of diffeomorphisms on $\bbR^d$. 
    
    Let $x\in \bbR^d$ and $\omega^*\in\Sigma^*$. The explicit form of $\Psi_{s,t}^{z_0} (\omega^*)x$ given by \eqref{eq:LNAflow} can be obtained by solving \eqref{eq:stochflowSDE}. 
\end{proof}

\section{The LNA and finite time stability}

In order to derive an explicit form for the LNA's maximal finite time Lyapunov exponent, we follow a standard argument used to define finite time Lyapunov exponents for dynamical systems and other flows.

\begin{proof}[Proof of Theorem \ref{mftlethm}]
    Let $x\in \bbR^d, \ \omega^*\in\Sigma^*, z_0\in\bbR^d $ and $T>0$. Let $\{z(t) \ | \ t\geq 0, \ z(0) = z_0\}$ be a solution of the RRE \eqref{eq:RRE} for some $z_0\in\bbR^2$. We will study the evolution of $x$ and an arbitrary point $y$ infinitesimally close to $x$. That is, $y=x+\delta x(0)$ where $\delta x(0)$ is infinitesimal and, for now, arbitrarily oriented. Under the LNA with RRE solution $\{z(t) \ | \ t\geq 0, \ z(0) = z_0\}$, this pertubation at time $T$ is given by,
     \[ \delta x(T) = \Psi_{0,T}^{z_0} (\omega^*)y-\Psi_{0,T}^{z_0} (\omega^*)x = \mathrm{D}_x \Psi_{0,T}^{z_0} (\omega^*)x \ \delta x(0) + o(\norm{\delta x(0)}^2) = C(0,T) \delta x(0) + o(\norm{\delta x(0)}^2) \]
     where the second equality follows from taking a Taylor expansion of $x\mapsto\Psi_{0,T}^{z_0} (\omega^*)x$ about the point $y$. From herein we assume that the term $o(\norm{\delta x(0)}^2)$ is negligible since $\delta x(0)$ is small. Now,
     \[ \norm{\delta x(T)} = \sqrt{\left(C(0,T) \delta x(0)\right)^\intercal C(0,T) \delta x(0)} = \sqrt{{\delta x(0)}^\intercal {C(0,T)}^\intercal C(0,T) \delta x(0)} = \sqrt{ {\delta x(0)}^\intercal \Delta \delta x(0)},  \]
     where we define $\Delta=\Delta(0,T):=C(0,T)^\intercal C(0,T)$. Recall that the dot product of two vectors is maximised when the two vectors are parallel. Therefore $\norm{\delta x(T)}$ attains its maximum when $\delta x(0)$ is parallel to the eigenvector $\overline{\delta x}(0)$ associated to the maximum eigenvalue of $\Delta$, $\lambda_{\text{max}}(\Delta)$. That is,
     \[ \max_{\delta x(0)}\norm{\delta x(T)} = \sqrt{ {\overline{\delta x}(0)}^\intercal \Delta \overline{\delta x}(0)} = \sqrt{\lambda_{\text{max}}(\Delta)}\norm{\overline{\delta x}(0)} = \norm{C(0,T)}\norm{\overline{\delta x}(0)}. \]
     The last equality holds since the matrix norm of $C(0,T)$ is induced by the Euclidean norm on $\bbR^d$. Consequently if,
     \[ \lambda^{z_0}_T = \lambda^{z_0}_T(\omega^*,x) = \frac{1}{T}\log\norm{C(0,T)}, \]
     then,
     \[ \max_{\delta x(0)}\norm{\delta x(T)} = e^{\lambda^{z_0}_T T}\norm{\overline{\delta x}(0)} \]
     and so $\lambda^{z_0}_T$ denotes the maximal average exponential rate of separation of infinitesimally close paths over the time interval $[0,T]$. Note that this rate is indeed independent of $x$ and $\omega^*$.
\end{proof}

\section{Discussion}

We have demonstrated how the approximate processes  $\{X^\CLE_n \ | \ n\in\bbNz\}$ and $\{X^\LNA(t) \ | \ t\geq 0\}$ can be described in such a way as to explore the dynamical properties of stochastic reaction network models. In doing so, we've proved interesting results about an oscillatory network whose deterministic process exhibits a Hopf bifurcation, which unlike other studies, has noise driven by the chemical physics of stochastic reaction networks. An interesting study could be compare its finite time dynamics to the deterministic Brusselator with random bifurcation parameter as in \cite{engel2023noiseinduced}. That is, fast and slow regimes that occur when one reaction rate is much larger then the other. We have also shown evidence of a limitation of models whose accuracy depends on repeated use of the LNA; finite time stability is lost in LNA predictions made in short time intervals.

\bibliographystyle{plain}
\bibliography{main} 
\end{document}